\documentclass[12pt,a4paper,twoside]{article}
\usepackage{amssymb}
\usepackage{amsmath}
\usepackage{amsthm}
\usepackage[latin1]{inputenc}
\usepackage{epic,url}
\usepackage{enumitem}
\usepackage{mdwlist}
\usepackage{xcolor}
\usepackage[title,titletoc,page]{appendix}

\usepackage{mathrsfs}
\DeclareFontFamily{U}{rsfs}{\skewchar\font127 }
\DeclareFontShape{U}{rsfs}{m}{n}{%
   <-6.5> rsfs5
   <6.5-8> rsfs7
   <8-> rsfs10
}{}

\title{\LARGE\bfseries A Differential Game with Symmetric Incomplete Information on Probabilistic Initial Condition and with Signal Revelation}

\date{}
\author{Xiaochi WU\\ \small Department of Mathematics, Shantou University\\ \small 5 Cuifeng Road, Shantou, 515821, People's Republic of China}
\newtheorem{mdef}{Definition}[section]
\newtheorem{mlem}{Lemma}[section]
\newtheorem{mprps}{Proposition}[section]

\newtheorem{mthm}{Theorem}[section]
\newtheorem{mcor}{Corollary}[section]
\newtheorem{mrmk}{Remark}[section]
\newtheorem{massum}{Assumptions}[section]
\numberwithin{equation}{section}

\DeclareMathOperator*{\R}{\mathbb{R}}
\DeclareMathOperator*{\N}{\mathbb{N}}

\DeclareMathOperator*{\U}{\mathbb{U}}
\DeclareMathOperator*{\V}{\mathbb{V}}

\newcommand{\txt}[1]{\text{\textnormal{#1}}}

\newcommand{\pr}[1]{\left( #1 \right)}

\newcommand{\abs}[1]{\left| #1 \right|}
\newcommand{\supp}[1]{\txt{supp }#1}
\begin{document}

%||Title Page||-------------------------------------------------------------
\clearpage
\maketitle

%============
%||Abstract||
%============

\bfseries \begin{center}Abstract:\end{center}
\mdseries In this paper, we investigate the existence and characterization of the value for a two-player zero-sum differential game with symmetric incomplete information on a continuum of initial positions and with signal revelation. Before the game starts, the initial position is chosen randomly according to a probability measure with compact support, and neither player is informed of the chosen initial position. However, they observe a public signal revealing the current state as soon as the trajectory of the dynamics hits a target set. We prove that, under a suitable notion of signal-dependent strategies, the value of the game exists, and the extended value function of the game is the unique viscosity solution of an associated Hamilton-Jacobi-Isaacs equation that satisfies a boundary condition.


%$$$$$$$$$$$$$$
%$$Main Part$$$
%$$$$$$$$$$$$$$
\section*{Introduction}
In this paper, we study a two-person zero-sum differential game of symmetric incomplete information with a signal-revealing mechanism. The dynamical system is given by:
\begin{equation}\label{dym_sys1}
	\left\{\begin{aligned}
		&\dot{x}(t)=f\pr{x(t),u(t),v(t)},&t\ge 0,\\
		&\dot{y}(t)=g\pr{y(t),u(t),v(t)},&t\ge 0,\\
		&\big(x(0),y(0)\big)=(x_0,y_0):=z_0\in{\R}^n\times{\R},
	\end{aligned}\right.
\end{equation}
where $u:\R_+\to\U$ and $v:\R_+\to\V$ are Lebesgue measurable maps with $\U$, $\V$ being compact metric spaces. Both $f:{\R}^n\times \U\times \V\to {\R}^n$ and $g:{\R}\times \U\times \V\to {\R}^+_*$ are assumed regular enough so that the above dynamics \eqref{dym_sys1} has a unique solution denoted by 
	$$t\mapsto \pr{X^{x_0,u,v}_t, Y^{y_0,u,v}_t},$$
To the initial position $z_0$ and the pair of admissible controls $(u,v)$ is associated the following running cost 
\begin{equation}\label{payoff1}
	J(z_0,u,v)=\int^\infty_0 e^{-\lambda t}\ell\pr{x(t),y(t),u(t),v(t)}dt.
\end{equation}
Here $\lambda>0$ and $\ell:{\R}^n\times \R\times \U\times \V\to\R$ is bounded, continuous.\par
Let us describe the game procedure. Let $\mu_0$ be a probability measure on ${\R}^{n+1}$ with compact support. The game $\mathcal G(\mu_0)$ with symmetric incomplete information and signal revealing is played as follows:
\begin{enumerate}
	\item
		Before the game begins, the initial position $z_0=(x_0,y_0)\in\mathbb R^{n+1}$ is chosen randomly according to the probability measure $\mu_0$, and the chosen initial data $z_0$ is communicated to neither player.
	\item
		During the game, Player 1 chooses the control $u$ and aims to minimize the cost $J(z_0,u,v)$, while Player 2 chooses the control $v$ and aims to maximize the same cost. Both players are assumed to observe all played actions with perfect memory during the game.
	\item
		If the trajectory $t\mapsto Y^{y_0,u,v}_t$ hits the target set $[M_0,\infty)$, the current state $Z^{z_0,u,v}_t$ is publicly announced to both players at the hitting time $$t=\mathcal{T}(y_0,u,v):=\inf\{t\ge 0\ |Y^{y_0,u,v}_t\ge M_0\}.$$
	\item
		The dynamic \eqref{dym_sys1}, running cost \eqref{payoff1} and the probability measure $\mu_0$ are common knowledge of both players.
\end{enumerate}
In game $\mathcal G(\mu_0)$, the second equation in system \eqref{dym_sys1} can be interpreted as the accumulation of public knowledge during the game, and the random variable $\max\pr{M_0-y_0,0}$ can be viewed as the required quantity (also unknown) of accumulated knowledge before the current state is revealed. As shown in \cite{Wu_DGAA19}, step 3 of the above game procedure is equivalent to the observation of the following signal process with perfect memory:
	\begin{equation*}
		s_{z,u,v}(t)=\left\{\begin{aligned}
			&0\in\mathbb{R}^{n+2},&\txt{if }t<\mathcal{T}(y_0,u,v),\\
			&(1,X^{x_0,u,v}_t, Y^{y_0,u,v}_t),&\txt{else.}
		\end{aligned}\right.
	\end{equation*}	 
An important feature of such a signal revelation mechanism is that during the game, the players will update their information about the unknown initial data even before the actual revelation of the current state. Indeed, if the players receive no public signal before some moment $t>0$, they will learn that $$y_0\notin \{y\in\mathbb{R}\ |\ \mathcal{T}(y_0,u,v)<t\}.$$

Differential games with asymmetric information of finite type were investigated in \cite{Cardaliaguet_2007} (see also \cite{Survey_diff,Cardaliaguet_2009,O_Barton15}), which generalizes the theory of repeated games with incomplete information introduced in \cite{A&M}. Cases with asymmetric information on a continuum of initial position were considered in \cite{C&J&Q14,J&Q&X16,J&Q18,C21}. Our game model consists in a generalization of the differential game with symmetric incomplete information studied in \cite{C&Q08}. The signal mechanism in our game model is inspired by those in repeated games with incomplete information in \cite{Forges82,Kohlb&Z74,N&S97,N&S98}, and differential games with symmetric information of finite type and with similar signal structure were treated in \cite{Wu_SICON18,Wu_DGAA19}. A differential game with one-sided partial observation of the current state was studied in \cite{B&R95}.  

Our goal is to prove that game $\mathcal{G}(\mu_0)$ has a value and to obtain a characterization of the value function as the unique viscosity solution of an appropriate Hamilton-Jacobi-Isaacs equation. As mentioned above, under a suitable notion of strategies, the information structure will change progressively as the game unfolds, and therefore we have to restrict ourselves to a suitable open set to prove a dynamic programming principle for the value functions. Moreover, the boundary condition for the associated Hamilton-Jacobi-Isaacs equation is not automatically verified by the value functions. Thus, we can no longer prove the existence of value for $\mathcal{G}(\mu_0)$ directly by showing that it is the unique viscosity solution to the HJI equation as in \cite{C&Q08,E&S,J&Q&X16,C21}. Accordingly, we first prove the existence of value via the approach employed in \cite{C&J&Q14,J&Q&X16}. We show that the value functions of game $\mathcal{G}(\mu_0)$ is continuous by rewriting them respectively into the upper and lower values of another differential game with symmetric incomplete information and a terminal cost at the controlled stopping time $\mathcal{T}(y_0,u,v)$. Then the existence of value is obtained by approximating the probability measure $\mu_0$ with a sequence $\{\mu_n\}_{n\in{\N}_*}$ with finite support, and by passing to the limit using the existence result obtained in \cite{Wu_DGAA19}. Finally, inspired by \cite{J&Q18}, we introduce the notion of extended value functions and we show that they are the unique viscosity solution satisfying a boundary condition and a set of regularity conditions to the HJI equation on an open set dependent on the proability measure $\mu_0$ and the signal structure.

The rest of this paper is organized as follows. After the preliminary section presenting useful notations and assumptions on the game model, we introduce proper notions of signal dependent strategies in section 2 and we study their properties before writing the game $\mathcal{G}(\mu_0)$ in normal form. In section 3, we prove the regularity property of the value functions, and the existence of value under Isaacs' condition is established. In section 4, we study the extended value functions of the game and prove a dynamic programming principle for the extended value functions. The last section is devoted to the characterization of the extended value function as the unique viscosity solution of the HJI equation.
\section{Preliminaries and Assumptions}\label{sect_PR}
In this paper, for any $m\in{\N}_*$, we denote by $\|x\|$ the Euclidean norm of $x\in\mathbb{R}^m$, and the scalar product of any $x,y\in\mathbb{R}^m$ is denoted by $x\cdot y$. The open ball with center $x\in{\R}^m$ and radius $r>0$ is denoted by $B(x;r)$. For $X\subset \mathbb{R}^m$, let $C(X,X)$ denote the set of continuous maps from $X$ to $X$. Meanwhile, the notation $C(X)$ stands for the set of continuous real-valued functions on $X$.
\subsection{Dynamics and Payoff}
Let $\U$, $\V$ be compact metric spaces endowed respectively with the corresponding Borel $\sigma$-algebra. We denote respectively by $\mathcal U$ (resp. $\mathcal V$) the sets of admissible (Lebesgue measurable) controls $u:{\R}_+\to\U$ (resp. $v:{\R}_+\to\V$). We assume that:
\begin{massum}\label{Assum1}
	\begin{enumerate}[label=\roman*)]
		\item
			$f:{\R}^n\times \U\times \V\to {\R}^n$ is bounded and continuous in all variables, and Lipschitz continuous in the state variable $x$ uniformly with respect to $(u,v)$;
		\item
			$g:{\R}\times \U\times \V\to {\R}^+$ is bounded and continuous in all variables, and Lipschitz continuous in the first variable $y$ uniformly with respect to $(u,v)$;
		\item
			$g(y,u,v)>0$ for all $(y,u,v)\in(-\infty,M_0]\times\U\times\V$;
		\item
			$\ell:{\R}^n\times \R\times \U\times \V\to\R$ is bounded and continuous in all variables, and Lipschitz continuous in the state variables $z=(x,y)$ uniformly with respect to $(u,v)$.
	\end{enumerate}
\end{massum}
To simplify notation, we write for all $z=(x,y)\in{\R}^{n+1}$ and $(u,v)\in\mathbb U\times\mathbb{V}$,
	$$F(z,u,v)=\big(f(x,u,v),g(y,u,v)\big).$$
The dynamical system \eqref{dym_sys1} can thus be written as
\begin{equation}\label{dym_sys2}
	\left\{\begin{aligned}
		&\dot{z}(t)=F\big(z(t),u(t),v(t)\big),&t\ge 0,\\
		&z(0)=z_0\in{\R}^{n+1}.
	\end{aligned}\right.
\end{equation}
with the payoff 
\begin{equation}\label{payoff2}
	J(z_0,u,v)=\int^\infty_0 e^{-\lambda t}\ell\big(Z^{z_0,u,v}_t,u(t),v(t)\big)dt.
\end{equation}
It follows from Assumptions \ref{Assum1} that $F:{\R}^{n+1}\times \U\times \V\to {\R}^n\times{\R}_+$ is bounded and continuous in all variables, and Lipschitz continuous in the state variable $z$ uniformly on $\mathbb{U}\times\mathbb{V}$ with its Lipschitz constant denoted by $L_F>0$. It's well-known that, under Assumptions \ref{Assum1}, given initial position $z_0=(x_0,y_0)\in{\R}^{n+1}$ and any pair of admissible controls $(u,v)\in\mathcal U\times\mathcal V$, there exists a unique solution to system \eqref{dym_sys2}, and we denote its trajectory by
$$t\mapsto Z^{z_0,u,v}_t:=\pr{X^{x_0,u,v}_t, Y^{x_0,u,v}_t}.$$ 
Hence, the hitting time $\mathcal T(y_0,u,v)$ of the trajectory $Y^{y_0,u,v}_\cdot$ at $[M_0,\infty)$ can be rewritten as
\begin{equation*}
	\mathcal{T}(z_0,u,v):=\inf\{t\ge 0\ |Z^{z_0,u,v}_t\in{\R}^n\times [M_0,\infty)\}=\mathcal{T}(y_0,u,v).
\end{equation*}
In addition, following standard estimations, for all $T>0$ and $(z,u,v)\in{\R}^{n+1}\times \mathcal{U}\times\mathcal{V}$,
\begin{align}\label{eq_assum_reg1}
	&\|Z^{z,u,v}_t-Z^{z'u,v}_t\|\le e^{L_F T}\|x-x'\|,\ \forall t\in[0,T]\\
	\label{eq_assum_reg2}&\|Z^{z,u,v}_t-Z^{z,u,v}_s\|\le \|F\|_{\infty}|t-s|,\ \forall t,s\ge 0.
\end{align}
Let us equipped the sets of admissible controls $\mathcal U$ and $\mathcal V$ respectively with $L^1_{loc}$-topology and the associated Borel $\sigma$-algebra. It is a classical result that the map $(u,v)\mapsto Z^{z_0,u,v}_t$ is continuous. We recall the the following regularity properties of the payoff $J(z_0,u,v)$ and the hitting time $\mathcal T(z_0,u,v)$.
\begin{mlem}[\cite{book_B&CD}]\label{lem_costcont}
	The function $(z,u,v)\mapsto J(z,u,v)$ is continuous. In particular, for all $(u,v)\in\mathcal U\times\mathcal V$, $z\mapsto J(z,u,v)$ is H\"{o}lder continuous independent of $(u,v)$ with exponent $\gamma$:
	$$\gamma=\left\{\begin{aligned}
	&1,&\text{if }L<\lambda,\\
	&\text{any }\gamma<1,&\text{if }L=\lambda,\\
	&\lambda/L,&\text{if }L>\lambda,
\end{aligned}\right.$$
where $L=\max(L_F,L_\ell)$.
\end{mlem}
\begin{mlem}[\cite{Wu_SICON18}]
	The map $(y,u,v)\mapsto \mathcal T(y,u,v)$ is continuous. Furthermore, for any pair of admissible controls $(u,v)\in\mathcal U\times\mathcal V$, the function $\mathcal T\pr{\cdot,u,v}$ is locally Lipschitz continuous. Namely, for any $S<M_0$, there exists $C_S>0$ such that for all $S\le y_0\le y_0'<M_0$ and any $(u,v)\in\mathcal U\times\mathcal V$,
	\begin{equation}\label{lip_hittime}
		\abs{\mathcal T\pr{y_0,u,v}-\mathcal T\pr{y_0',u,v}}\le C_S\abs{y_0-y_0'}.
	\end{equation}
\end{mlem}
As a direct consequence of \eqref{eq_assum_reg1}, \eqref{eq_assum_reg2}, and the above lemma, for any $(u,v)\in\mathcal{U}\times\mathcal{V}$, the map
$(z_0,u,v)\mapsto Z^{z_0,u,v}_{\mathcal T(z_0,u,v)}$ is Borel measurable.

\subsection{Probability Distributions on the Initial States}
In this subsection, we present several useful notations, tools and results about probability measures and the theory of optimal transport. We denote by $\mathcal P(\mathbb{R}^{n+1})$ the set of Borel probability measures $\mu$ on $\mathbb{R}^{n+1}$ with compact support. 

It is well-known that $\mathcal P(\mathbb{R}^{n+1})$ can be endowed with the Wasserstein distance $W_2$:
$$W_2(\mu,\nu)=\min_{\pi\in\Pi(\mu,\nu)}\Big\{\int_{Z^2}\|z-z'\|^2d\pi(z,z')\Big\}^\frac{1}{2},\ \forall \mu,\nu\in \mathcal P(\mathbb{R}^{n+1}).$$
Here $\Pi(\mu,\nu)$ denote the set of probability measures on $\mathbb{R}^{n+1}\times\mathbb{R}^{n+1}$ with $\mu$ as its first marginal and $\nu$ its second marginal. For all $\mu\in\mathcal{P}(\mathbb{R}^{n+1})$, the push-forward measure $\Phi\sharp\mu$ of $\mu$ by the Borel measurable map $\Phi:\mathbb{R}^{n+1}\to\mathbb{R}^{n+1}$ is define by
$$\Phi\sharp\mu(A)=\mu(\Phi^{-1}(A)),\forall A\in\mathcal{B}(\mathbb{R}^{n+1}).$$

Let the compact subset $Z=X\times [0,M_0+\eta]\subset{\R}^{n+1}$ where $X\subset{\R}^n$ is compact and $\eta>0$. We denote by $\Delta(Z)$ the set of probability measures on $Z$.  We denote by $\Delta(Z)$ the set of Borel probability measures on $Z$, and we equipped $\Delta(Z)$ with the $W_2$-distance. We refer interested readers to \cite{book_Villani} for classic results of the theory of optimal transport and Wasserstein distance. For any $\Phi,\Psi\in C(Z,Z)$ and $\mu\in \Delta(Z)$, we denote
\begin{align*}
	\|\Phi\|_{L^2_\mu}&=\Big(\int_Z \Phi^2(z)d\mu(z)\Big)^\frac{1}{2};\\
	\langle \Phi,\Psi\rangle_{L^2_\mu}&=\int_Z \Phi(z)\cdot\Psi(z)d\mu(z).
\end{align*}

\subsection{Isaacs's Condition}
Let us denote, for any $z=(x,r)\in\mathbb R^{n+1}$, $\xi\in\mathbb R^n$, and $\eta\in\mathbb R$,
	\begin{align*}
		H^+(z,\zeta)=&\inf_{u\in\U}\sup_{v\in\V}\big\{ F(z,u,v)\cdot\zeta+\ell(z,u,v)\big\},\\
		H^-(z,\zeta)=&\sup_{v\in\V}\inf_{u\in\U}\big\{ F(z,u,v)\cdot\zeta +\ell(z,u,v)\big\}.
	\end{align*}
	Let us recall the following Isaacs' condition \eqref{eq_IC1}, which is often assumed in the literature for establishing the existence of value for zero-sum differential games with complete information (for cases without Isaacs' condition, see \cite{B&L&Q13} and \cite{B&Q&R&X15}). 
	\begin{equation}\label{eq_IC1}
			\forall (z,\zeta)\in\mathbb R^{n+1}\times\mathbb R^{n+1},\ H^+(z,\zeta)=H^-(z,\zeta).
	\end{equation}
	In this paper, we assume the following two different version of Isaacs' condition:
	\begin{enumerate}[label=(IC\arabic*)]
		\item $\forall (\mu,p)\in\mathcal P(\mathbb R^{n+1})\times C(\mathbb{R}^{n+1},\mathbb{R}^{n+1})$,
			\begin{equation}\label{eq_IC2}
				\begin{aligned}
					&\inf_{u\in\mathbb{U}}\sup_{v\in\mathbb V}\int_{\mathbb{R}^{n+1}}\big\{F(z,u,v)\cdot p(z)+\ell(z,u,v)\big\}d\mu(z)\\
					=&\sup_{v\in\mathbb V}\inf_{u\in\mathbb{U}}\int_{\mathbb{R}^{n+1}}\big\{F(z,u,v)\cdot p(z)+\ell(z,u,v)\big\}d\mu(z).
				\end{aligned}
			\end{equation}
		\item $\forall (\mu,p)\in\Delta(Z)\times C(Z,\mathbb{R}^{n+1})$,
			\begin{equation}\label{eq_IC3}
				\begin{aligned}
					&\inf_{u\in\mathbb{U}}\sup_{v\in\mathbb V}\int_{Z}\big\{F(z,u,v)\cdot p(z)+\ell(z,u,v)\big\}d\mu(z)\\
					=&\sup_{v\in\mathbb V}\inf_{u\in\mathbb{U}}\int_{Z}\big\{F(z,u,v)\cdot p(z)+\ell(z,u,v)\big\}d\mu(z).
				\end{aligned}
			\end{equation}
	\end{enumerate}
We recall the following result from \cite{J&Q18} (Proposition 1):
\begin{mlem}
	The conditions below are equivalent to the Isaacs' condition \eqref{eq_IC3}:
	\begin{enumerate}[label=(IC\arabic*)]
		\item $\forall I\in{\N}_*$, $\forall \mu=\sum^I_{i=1}q_i\delta_{z_i}\in\Delta(Z)$, and $\forall p=(p_1,...,p_I)\in\mathbb{R}^{(n+1)I}$,
			\begin{equation}\label{eq_IC4}
				\begin{aligned}
					&\inf_{u\in\mathbb{U}}\sup_{v\in\mathbb V}\sum_{i=1}^I q_i \big[F(z_i,u,v)\cdot p_i+\ell(z_i,u,v)\big]\\
					=&\sup_{v\in\mathbb V}\inf_{u\in\mathbb{U}}\sum_{i=1}^I q_i \big[F(z_i,u,v)\cdot p_i+\ell(z_i,u,v)\big].
				\end{aligned}
			\end{equation}
		\item $\forall (\mu,\Phi,p)\in\Delta(Z)\times C(Z,Z)\times C(Z,\mathbb{R}^{n+1})$,
			\begin{equation}\label{eq_IC5}
				\begin{aligned}
					\mathcal{H}(\mu,\Phi,p):=&\inf_{u\in\mathbb{U}}\sup_{v\in\mathbb V}\int_{Z}\big\{F(\Phi(z),u,v)\cdot p(z)+\ell(\Phi(z),u,v)\big\}d\mu(z)\\
					=&\sup_{v\in\mathbb V}\inf_{u\in\mathbb{U}}\int_{Z}\big\{F(\Phi(z),u,v)\cdot p(z)+\ell(\Phi(z),u,v)\big\}d\mu(z).
				\end{aligned}
			\end{equation}
	\end{enumerate}
\end{mlem}

\section{Strategies and Value Functions}\label{sect_SV}
In this section, we introduce the definition of signal-dependent non-anticipative strategies with delay for game $\mathcal{G}(\mu_0)$. The notion of non-anticipative strategies with delay (in short, NAD strategies) for differential games was studied in \cite{C&Q08}, and signal-dependent NAD strategies for differential games with signal revelation were introduced in \cite{Wu_SICON18} and \cite{Wu_DGAA19} regarding different types of signal functions. By playing signal-dependent strategies, the players choose their actions according to the presence or the absence of signal revelation. In addition, the property of NAD strategies (see Lemma \ref{SNAD_lem2} of Section 2.1) allows us to formulate the game in normal form and thereby define its value functions.
\subsection{Strategies}
We first recall the notion of NAD strategy (cf. \cite{C&Q08}):
\begin{mdef}[NAD strategy]\label{Def_NAD}
	An NAD strategy for Player 1 in game $\mathcal{G}(\mu_0)$ is a Borel measurable map $\alpha:\mathcal V\to\mathcal U$ such that there eixtss $\tau>0$, for any $v_1,v_2\in\mathcal V$ and $t\ge 0$, if $v_1=v_2$ a.e. on $[0,t]$, then
	$$\alpha(v_1)=\alpha(v_2) \txt{ a.e. on }[0,t+\tau].$$
	NAD strategies for Player 2 are defined symmetrically.
\end{mdef}
We denote by $\mathcal A_d$ (resp. $\mathcal B_d$) the set of NAD strategies for Player 1 (resp. Player 2). An important property of NAD strategies is stated in the following:
\begin{mlem}[\cite{C&Q08}]\label{SNAD_lem1B}
	Given a pair of NAD strategies $(\alpha,\beta)\in\mathcal{A}_d\times\mathcal{B}_d$, there exists a unique pair of admissible controls $(u_{\alpha\beta},v_{\alpha\beta})\in\mathcal U\times\mathcal V$ such that
	$$\alpha(v_{\alpha\beta})=u_{\alpha\beta}\txt{ and }\beta(u_{\alpha\beta})=v_{\alpha\beta}.$$
\end{mlem}
In other words, if the players choose the pair of NAD strategies $(\alpha,\beta)$, then the associated pair of admissible controls $(u_{\alpha\beta},v_{\alpha\beta})$ will be played.
With the signal mechanism in game $\mathcal G(\mu_0)$, players should employ signal-dependent NAD strategies (SNAD strategies). The following notion of signal-dependent NAD strategies is adapted from \cite{C&J&Q14} and \cite{Wu_SICON18}.
\begin{mdef}[SNAD strategy]\label{Def_SNAD}
	An SNAD strategy for Player 1 in game $\mathcal{G}(\mu_0)$ is a Borel measurable map $A:\R_+\times{\R}^{n+1}\times\mathcal V\to\mathcal U$ such that:
	\begin{enumerate}[label=\roman*)]
		\item
			$\exists \tau>0$, $\forall T_1,T_2\in \R_+$, $z_1,z_2\in{\R}^{n+1}$ and $v\in\mathcal V$,
			$$A(T_1,z_1,v)=A(T_2,z_2,v) \txt{ a.e. on }[0,T_1\wedge T_2+\tau].$$
		\item
			$\forall (T,z)\in \R_+\times{\R}^{n+1}$, $ v_1,v_2\in\mathcal V$ and $t\ge 0$, if $v_1=v_2$ a.e. on $[0,t]$, then
			$$A(T,z,v_1)=A(T,z,v_2) \txt{ a.e. on }[0,t+\tau].$$
	\end{enumerate}	 
	SNAD strategies for Player 2 are defined symmetrically.
\end{mdef}
We denote by $\mathcal A_s$ the set of SNAD strategies for Player 1, and $\mathcal B_s$ denotes the set of SNAD strategies for Player 2.
\begin{mrmk}
	 It is clear that NAD strategies can be viewed as SNAD strategies. Thus,  we have $\mathcal{A}_d\subset\mathcal{A}_s$ and $\mathcal{B}_d\subset\mathcal{B}_s$.
\end{mrmk}
\begin{mrmk}
	 By playing an SNAD strategy $A\in\mathcal{A}_s$, Player 1 will employ the admissible control $u=A(T,z,v)$ in a non-anticipative manner if the data $z$ is publicly revealed at the moment $t=T$ and his/her adversary plays the control $v$. In this regard, Condition i) in Definition \ref{Def_SNAD} is essential. Indeed, players' actions should not depend on data that has not yet been revealed. In fact, by the lemma below, under SNAD strategies, the players follow a (unique) associated NAD strategy until after the current data is revealed.
\end{mrmk}

\begin{mlem}\label{SNAD_lem1}
	For all $A\in\mathcal{A}_s$, there exists a unique $\alpha_A\in \mathcal{A}_d$ such that, $\forall(T,z)\in\R_+\times {\R}^{n+1}$, $\forall v\in\mathcal V$, and $\forall \tau>0$ a delay of $A$,
\begin{equation}\label{eq_SNAD_lem1}
	\alpha_A(v)=A(T,z,v),\txt{ a.e. on }[0,T+\tau].
\end{equation}
	Symmetrically, for any $B\in\mathcal{B}_s$, there exists a unique $\beta_B\in \mathcal{B}_d$ such that, $\forall(T,z)\in\R_+\times {\R}^{n+1}$, $\forall u\in\mathcal U$, and $\forall \tau>0$ a delay of $B$,
\begin{equation}\label{eq_SNAD_lem2}
	\beta_B(u)=B(T,z,u),\txt{ a.e. on }[0,T+\tau].
\end{equation}
\end{mlem}
\begin{proof}
	We only prove the first claim and the second can be verified in a symmetrical manner. Given any SNAD strategy $A\in\mathcal A_s$ and its delay $\tau>0$, we define a sequence of maps $\{\alpha_n:\mathcal V\to\mathcal U\}_{n\in\mathbb N_*}$, by setting for all $n\in\N$ and some constant action parameter $u_0\in\mathbb U$:
$$\alpha_n(v)(t)=\left\{
\begin{aligned}
	&A(n,\textbf{0},v)(t), &t\in[0,n],\\
	&u_0, &t>n,
\end{aligned}\right.$$
	It is clear that $\alpha_n(v)\in\mathcal U$, and as the composition of measurable maps, $\alpha_n:\mathcal V\to\mathcal U$ is Borel measurable. Furthermore, since $A(n,\textbf{0},\cdot)\in\mathcal A_d$, for any $v,v'\in\mathcal V$ verifying $v=v'$ a.e. on $[0,t]$ for some $t>0$, one has
$$\alpha_n(v)=\alpha_n(v')\text{ a.e. on }[0,(t+\tau)\wedge n]$$
and $a_n(v)=a_n(v')$ on $[n,+\infty)$. Thus $\alpha_n\in\mathcal A_d$, $\forall n\in\mathbb{N}$.

	Let us define for all $v\in\mathcal V$, a control $\alpha_A(v)$ by setting $\alpha_A(v)(t)=\alpha_n(v)(t)$, for all $n\in\N$ and $t\in[n,n+1)$. One can check clear that the control $\alpha_A(v)$ is well defined and admissible. Since $\alpha_n(v)$ and $\alpha_m(v)$ coincide a.e. on $[0,n\wedge m]$, and for any $N\in\mathbb N_*$
$$\lim_{n\to\infty}\int^N_0 d_\mathbb U\big(\alpha_n(v)(t),\alpha_A(v)(t)\big)dt=0.$$
	Thus, as the point-wise limit of measurable maps $\alpha_n$, $\alpha_A:\mathcal{V}\to\mathcal{U}$ is Borel measurable. In addition, $\forall v,v'\in\mathcal V$ and $t>0$, if $v=v'$ a.e. on $[0,t]$, one has, with $n(t)=\lceil t+\tau\rceil$,
$$\alpha_A(v)=\alpha_{n(t)}(v)=\alpha_{n(t)}(v')=\alpha_A(v'),\text{ a.e. on }[0,t+\tau].$$
	Consequently, $\alpha_A\in\mathcal A_d$, and $\forall(T,z,v)\in \R_+\times{\R}^{n+1}\times\mathcal{V}$,
$$\alpha_A(v)=\alpha_{n(T)}(v)=A(n(T),\textbf{0},v)=A(T,z,v)\text{ a.e. on }[0,T+\tau].$$
	It remains thus to check the uniqueness. If $\alpha_A$, $\alpha_A'\in\mathcal A_d$ are NAD strategies for Player 1 both verifying \eqref{eq_SNAD_lem1}, we have for all $N\in\N$ and for all $v\in\mathcal{V}$,
	$$\alpha_A(v)=A(N,\textbf{0},v)=\alpha_A'(v)\text{ a.e. on }[0,N].$$
	Therefore, for all $v\in\mathcal V$, $\alpha_A(v)=\alpha_A'(v)$ a.e. on $\R_+$ . The proof is complete.
\end{proof}
\begin{mrmk}
	An SNAD strategy $A\in\mathcal A_s$ for Player 1 in game $\mathcal G(\mu_0)$ can be viewed as being composed of the NAD strategy $\alpha_A$ and a collection of NAD strategies $(\alpha_{T,z})\in \mathcal A_d^{\R_+\times{\R}^{n+1}}$ with
	$$\alpha_{T,z}(v)=A(T,z,v)(\cdot+T),\ \forall v\in \mathcal{V}.$$
Heuristically, by choosing an SNAD strategy in game $\mathcal G(\mu_0)$, the player first plays the associated NAD strategy until the signal revelation, and he or she then chooses a new NAD strategy from the collection for the sub-game with complete information. 
\end{mrmk}
	As Lemma \ref{SNAD_lem2} below indicates, SNAD strategies possess a similar property to NAD strategies, which allows us to write game $\mathcal G(\mu_0)$ in normal form.
\begin{mlem}\label{SNAD_lem2}
	Given a pair of SNAD strategies $(A,B)\in\mathcal{A}_s\times\mathcal{B}_s$, for any $z_0=(x_0,y_0)\in\R^{n+1}$, there exists a unique pair of admissible controls $(u_{z_0},v_{z_0})\in\mathcal U\times\mathcal V$ such that
	$$A\Big(\mathcal{T}(z_0,u_{z_0},v_{z_0}),Z^{z_0,u_{z_0},v_{z_0}}_{\mathcal{T}(z_0,u_{z_0},v_{z_0})},v_{z_0}\Big)=u_{z_0}\txt{ and }B\Big(\mathcal{T}(z_0,u_{z_0},v_{z_0}),Z^{z_0,u_{z_0},v_{z_0}}_{\mathcal{T}(z_0,u_{z_0},v_{z_0})},u_{z_0}\Big)=v_{z_0}.$$
	Moreover, the map $z_0\mapsto (u_{z_0},v_{z_0})$ is Borel measurable.
\end{mlem}
\begin{proof}
	By Lemma \ref{SNAD_lem1}, there exists a unique pair of NAD strategies $(\alpha_A,\beta_B)\in\mathcal A_d\times \mathcal B_d$ such that for any $(T,z)\in{\R}_+\times\R^{n+1}$,
	$$\forall(u,v)\in\mathcal U\times\mathcal V,\ A(T,z,v)=\alpha_A(v), B(T,z,u)=\beta_B(u)\text{ a.e. on }[0,T+\tau].$$
	Here $\tau>0$ is the common delay of $A$ and $B$. Moreover, by Lemma \ref{SNAD_lem1B}, there exists a unique pair of admissible controls $(u_{AB},v_{AB})$ such that
	$$\alpha(v_{AB})=u_{AB}\text{ and }\beta(u_{AB})=v_{AB}\text{ a.e. on }\mathbb R_+$$	
	Let us denote by $\mathcal T(z_0)$ the hitting time $\mathcal T(z_0,u_{AB},v_{AB})$ when the pair of NAD strategies $(\alpha_A,\beta_B)$ is employed by the players in game $\mathcal{G}(\mu_0)$. It follows that for any $z_0,z\in\mathbb R^{n+1}$,
	$$A(\mathcal T(z_0),z,v_{AB})=u_{AB},\ B(\mathcal T(z_0),z,u_{AB})=v_{AB},\text{ a.e. on }[0,\mathcal T(z_0)+\tau].$$
	It follows again from Lemma \ref{SNAD_lem1B} (applied to $(A(T,z,\cdot),B(T,z,\cdot))\in\mathcal{A}_d\times\mathcal{B}_d$) that $(u_{AB},v_{AB})$ is the unique pair of admissible controls with this property. In other words, if the pair of SNAD strategy $(A,B)$ is played in game $\mathcal G(\mu_0)$, the revelation time $\mathcal T$ and the revealed signal are determined by the continuous map:
	\begin{align*}
		z_0\mapsto(\mathcal T(z_0),Z^{z_0,u_{AB},v_{AB}}_{\mathcal T(z_0)}).
	\end{align*}
	For simplicity, we write $Z^{z_0}_{\mathcal{T}}:=Z^{z_0,u_{AB},v_{AB}}_{\mathcal T(z_0)}$ for the rest of the proof.
	\paragraph*{} 
	To prove the lemma, we further extend the pair $(u_{AB},v_{AB})\big|_{[0,\mathcal T(z_0)+\tau]}$ to $\mathbb{R}_+$. More precisely, we prove by induction that, for all $n\in\N$, there exists a pair of admissible controls $(u_{z_0}^n,v^n_{z_0})$ such that the following equations hold almost everywhere on $[0,\mathcal T(z_0)+n\tau]$:
	\begin{equation}\label{eq_SNAD_lem2A}
		A(\mathcal T(z_0),Z^{z_0}_{\mathcal{T}},v^n_{z_0})=u^n_{z_0}\txt{ and }B(\mathcal T(z_0),Z^{z_0}_{\mathcal{T}},u^n_{z_0})=v^n_{z_0}
	\end{equation}
	Let us fix an arbitrary pair of control parameters $(u_0,v_0)\in\mathbb{U}\times \mathbb{V}$. We define $(u^1_{z_0},v^1_{z_0})\in\mathcal U\times \mathcal V$ by setting
	\begin{align*}
		&u^1_{z_0}(t)=\left\{\begin{aligned}
			&u_{AB}(t),&t\in[0,\mathcal T(z_0)+\tau],\\
			&u_0,&\text{else.}
		\end{aligned}\right.\\
		&v^1_{z_0}(t)=\left\{\begin{aligned}
			&v_{AB}(t),&t\in[0,\mathcal T(z_0)+\tau],\\
			&v_0,&\text{else.}
		\end{aligned}\right.
	\end{align*}
	By the NAD property of $A$ and $B$, it follows that 
	$$\begin{aligned}
		A(\mathcal T(z_0),Z^{z_0}_{\mathcal{T}},v^1_{z_0})=A(\mathcal 				T(z_0),Z^{z_0}_{\mathcal{T}},v_{AB})=u_{AB}=u^1_{z_0},\text{ a.e. on }[0,\mathcal T(z_0)+\tau],\\ B(\mathcal T(z_0),Z^{z_0}_{\mathcal{T}},u^1_{z_0})=B(\mathcal T(z_0),Z^{z_0}_{\mathcal{T}},u_{AB})=v_{AB}=v^1_{z_0},\text{ a.e. on }[0,\mathcal T(z_0)+\tau].
	\end{aligned}$$
	\paragraph*{}
	Assume that for some $k\ge 1$, there exists a pair of admissible controls $(u_{z_0}^k,v^k_{z_0})$ such that
	\begin{align*}
		A(\mathcal T(z_0),Z^{z_0}_{\mathcal{T}},v^k_{z_0})=u^k_{z_0}, B(\mathcal T(z_0),Z^{z_0}_{\mathcal{T}},u^k_{z_0})=v^k_{z_0},\text{ a.e. on }[0,\mathcal T(z_0)+k\tau].
	\end{align*}
	Let us construct the pair of controls $(u^{k+1}_{z_0},v^{k+1}_{z_0})\in\mathcal U\times \mathcal V$ by setting
	\begin{align*}
		&u^{k+1}_{z_0}(t)=\left\{\begin{aligned}
			&A(\mathcal T(z_0),Z^{z_0}_{\mathcal{T}},v^k_{z_0})(t),&t\in[0,\mathcal T(z_0)+(k+1)\tau],\\
			&u_0,&\text{else,}
		\end{aligned}\right.\\
		&v^{k+1}_{z_0}(t)=\left\{\begin{aligned}
			&B(\mathcal T(z_0),Z^{z_0}_{\mathcal{T}},u^k_{z_0})(t),&t\in[0,\mathcal T(z_0)+(k+1)\tau],\\
			&v_0,&\text{else.}
		\end{aligned}\right.
	\end{align*}
	By the above construction of $(u^{k+1}_{z_0},v^{k+1}_{z_0})$ and the NAD property of $A$ and $B$, one has $(u^{k+1}_{z_0},v^{k+1}_{z_0})=(u^{k}_{z_0},v^{k}_{z_0})$ a.e. on $[0,\mathcal T(z_0)+k\tau]$, and in addition,
	\begin{align*}
		A(\mathcal T(z_0),Z^{z_0}_{\mathcal{T}},v^{k+1}_{z_0})&=A(\mathcal T(z_0),Z^{z_0}_{\mathcal{T}},v^{k}_{z_0})=u^{k+1}_{z_0},\txt{ a.e. on }[0,\mathcal T(z_0)+(k+1)\tau];\\ 
		B(\mathcal T(z_0),Z^{z_0}_{\mathcal{T}},u^{k+1}_{z_0})&=B(\mathcal T(z_0),Z^{z_0}_{\mathcal{T}},u^{k}_{z_0})=v^{k+1}_{z_0},\txt{ a.e. on }[0,\mathcal T(z_0)+(k+1)\tau].
	\end{align*}
	Therefore equations \eqref{eq_SNAD_lem2A} hold for all $n\in{\N}_*$. Let us define
	\begin{align*}
		u_{z_0}(t)&=\left\{\begin{aligned}
			&u_{AB}(t),&t\in[0,\mathcal T(z_0)),\\
			&u^{n}_{z_0}(t),&t\in[T(z_0)+(n-1)\tau,\mathcal T(z_0)+n\tau),
		\end{aligned}\right.\\
		v_{z_0}(t)&=\left\{\begin{aligned}
			&v_{AB}(t),&t\in[0,\mathcal T(z_0)),\\
			&v^{n}_{z_0}(t),&t\in[T(z_0)+(n-1)\tau,\mathcal T(z_0)+n\tau).
		\end{aligned}\right.
	\end{align*}
	By our construction, $(u_{z_0},v_{z_0})=(u^n_{z_0},v^n_{z_0})$ a.e. on $[0,\mathcal T(z_0)+n\tau]$, for all $n\in \N$, and one can check that $(u_{z_0},v_{z_0})\in\mathcal U\times\mathcal V$.

	The uniqueness of $(u_{z_0},v_{z_0})\in\mathcal U\times\mathcal V$ verifying \eqref{eq_SNAD_lem2A} for all $n\in\mathbb N^*$ follows from Lemma \ref{SNAD_lem1B} and from the fact that $(A(\mathcal T(z_0),Z^{z_0}_{\mathcal{T}},\cdot),B(\mathcal T(z_0),Z^{z_0}_{\mathcal{T}},\cdot))\in\mathcal A_d\times\mathcal B_d$.

	To prove the Borel measurability of $z_0\mapsto(u_{z_0},v_{z_0})$, we notice that the map 
	$$z_0\mapsto (\mathcal{T}(z_0),Z^{z_0,u_{AB},v_{AB}}_{\mathcal T(z_0)},u_{AB},v_{AB})$$
	is continuous. Then, one can show by induction on $n\in\mathbb N^*$ that, as finite composition of continuous maps and measurable maps, $z_0\mapsto (u^{n}_{z_0},v^n_{z_0})$ is measurable for all $n\in{\N}_*$. But by the definition of $(u_{z_0},v_{z_0})$, we have
	$$(u^{n}_{z_0},v^n_{z_0})\xrightarrow{L^1_{loc}}(u_{z_0},v_{z_0}).$$ 
	Thus, as the pointwise limit of measurable maps, the map $z_0\mapsto (u_{z_0},v_{z_0})$ is also measurable. The proof is complete.
\end{proof}

\subsection{Value Functions}
	In view of Lemma \ref{SNAD_lem2} from the previous subsection, for any pair of SNAD strategies $(A,B)\in\mathcal A_s\times\mathcal B_s$, we denote by $(A,B)$ the Borel measurable map 
\begin{align*}
	(A,B):\mathbb R^{n+1}&\to\mathcal U\times\mathcal V\\
	z_0&\mapsto (u_{z_0},v_{z_0})
\end{align*}
	We are able to associate to each pair of SNAD strategies $(A,B)$ and initial state $z_0\in\mathbb R^{n+1}$ the payoff:
\begin{equation*}
		J(z_0,A,B)=J(z_0,u_{z_0},v_{z_0}).
\end{equation*}
	Moreover, for any $\mu_0\in\mathcal{P}(\mathbb{R}^{n+1})$, we can write the game $\mathcal G(\mu_0)$ in normal form by associating to each pair of SNAD strategies $(A,B)$ the expectation of cost:
\begin{equation*}
		J(\mu_0,A,B)=\int_{{\R}^{n+1}}J(z,(A,B)(z))d\mu_0(z),
\end{equation*}
	The upper and lower values of game $\mathcal G(\mu_0)$ are thus given by
\begin{equation}\label{eq_NF_VF}
	V^+(\mu_0)=\inf_{A\in\mathcal A_s}\sup_{B\in\mathcal B_s}J(\mu_0,A,B)\text{ and }V^-(\mu_0)=\sup_{B\in\mathcal B_s}\inf_{A\in\mathcal A_s}J(\mu_0,A,B).
\end{equation}
	It follows immediately from the above definition of $V^\pm$ that $V^+\ge V^-$. In particular, if $\mu_0=\delta_{z_0}$ is the Dirac mass at $z_0\in\mathbb R^{n+1}$, we write $V^\pm(z_0):=V^\pm(\delta_{z_0})$. Let us show that in this case, the value functions of $\mathcal{G}(\delta_{z_0})$ coincide with those of the corresponding differential game with complete information.
\begin{mlem}\label{lem_value1}
	For any $z_0\in\mathbb{R}^{n+1}$, 
	\begin{align}
		V^+(z_0)&=\inf_{\alpha\in\mathcal{A}_d}\sup_{v\in\mathcal{V}}J(z_0,\alpha(v),v);\\ \label{eq_lem_value1}
		V^-(z_0)&=\sup_{\beta\in\mathcal{B}_d}\inf_{u\in\mathcal{U}}J(z_0,u,\beta(v)).
	\end{align}
\end{mlem}
\begin{proof}
	We only prove the first equation \eqref{eq_lem_value1}. Let $\bar\alpha\in\mathcal{A}_d\subset\mathcal{A}_s$ be an $\varepsilon$-optimal strategy for the right-hand side of \eqref{eq_lem_value1}, i.e.
	\begin{equation}
		\inf_{\alpha\in\mathcal{A}_d}\sup_{v\in\mathcal{V}}J(z_0,\alpha(v),v)+\varepsilon\ge \sup_{v\in\mathcal{V}}J(z_0,\bar\alpha(v),v).
	\end{equation}
	By Lemma \ref{SNAD_lem2}, for any $B\in\mathcal{B}_s$, we have 
	\begin{equation}
		J(z_0,(\bar\alpha,B)(z_0))=J(z_0,\bar{\alpha}(v_{z_0}),v_{z_0})
		\le \sup_{v\in\mathcal{V}}J(z_0,\bar{\alpha}(v),v)\le \inf_{\alpha\in\mathcal{A}_d}\sup_{v\in\mathcal{V}}J(z_0,\alpha(v),v)+\varepsilon.
	\end{equation}
	where $v_{z_0}$ are defined as in Lemma \ref{SNAD_lem2}. Taking the supremum over $B\in\mathcal{B}_s$ on both sides of the last inequality above yields
	\begin{equation}
		V^+(z_0)\le \sup_{B\in\mathcal{B}_s}J(z_0,(\bar\alpha,B)(z_0))\le \inf_{\alpha\in\mathcal{A}_d}\sup_{v\in\mathcal{V}}J(z_0,\alpha(v),v)+\varepsilon.
	\end{equation}
	Since $\varepsilon>0$ is arbitrary, let $\varepsilon\to 0+$, and we obtain 
	\begin{equation}
		V^+(z_0)\le\inf_{\alpha\in\mathcal{A}_d}\sup_{v\in\mathcal{V}}J(z_0,\alpha(v),v).
	\end{equation}
	To check the opposite inequality, let $\bar{A}\in\mathcal{A}_s$ be an $\varepsilon$-optimal strategy for $V^+(z_0)$. Since $\mathcal{V}\subset\mathcal{B}_d\subset\mathcal{B}_s$, let $\alpha_{\bar A}$ and $u_{z_0}$ be defined respectively as in Lemma \ref{SNAD_lem1} and Lemma \ref{SNAD_lem2}, and we have, 
	\begin{equation}\label{eq_lem_value2}
		\begin{aligned}
			V^+(z_0)+\varepsilon\ge &\sup_{B\in\mathcal{B}_s}J(z_0,(\bar{A},B)(z_0))\\
			\ge &\sup_{v\in\mathcal{V}}J\big(z_0,\bar{A}\big(\mathcal{T}(z_0,\alpha_{\bar A}(v),v),Z^{z_0,\alpha_{\bar A}(v),v}_{\mathcal{T}(z_0,\alpha_{\bar A}(v),v)},v\big),v\big).
		\end{aligned}
	\end{equation}
	It suffices thus to check that the map 
	$$v\mapsto \hat\alpha(v):=\bar{A}\big(\mathcal{T}(z_0,\alpha_{\bar A}(v),v),Z^{z_0,\alpha_{\bar A}(v),v}_{\mathcal{T}(z_0,\alpha_{\bar A}(v),v)},v\big)$$
	belongs to $\mathcal{A}_d$. As the composition of the SNAD strategy $\bar A$ and the measurable map 
	$$v\mapsto \Big(\mathcal{T}(z_0,\alpha_{\bar A}(v),v),Z^{z_0,\alpha_{\bar A}(v),v}_{\mathcal{T}(z_0,\alpha_{\bar A}(v),v)},v\Big),$$ 
	it is clear $\hat{\alpha}:\mathcal{V}\to\mathcal{U}$ is measurable. Fix any $v,v'\in\mathcal{V}$. Let $\tau>0$ be a delay of $\bar A$, and let us assume that $v=v'$ a.e. on $[0,t]$ for some $t\ge 0$.\\
	If $t<T:=\min(\mathcal{T}(z_0,\alpha_{\bar A}(v),v),\mathcal{T}(z_0,\alpha_{\bar A}(v'),v'))$, one has
	\begin{equation}
		\hat{\alpha}(v)=\alpha_{\bar A}(v)=\alpha_{\bar A}(v')=\hat{\alpha}(v),\txt{ a.e. on }[0,t+\tau].
	\end{equation}
	Otherwise, if $t\ge T$, then $\mathcal{T}(z_0,\alpha_{\bar A}(v),v)=\mathcal{T}(z_0,\alpha_{\bar A}(v'),v')$ and the equation below holds a.e. on $[0,t+\tau]$:
	\begin{equation}
		\begin{aligned}
			\hat{\alpha}(v)&=\bar{A}\big(\mathcal{T}(z_0,\alpha_{\bar A}(v),v),Z^{z_0,\alpha_{\bar A}(v),v}_{\mathcal{T}(z_0,\alpha_{\bar A}(v),v)},v\big)\\
			&=\bar{A}\big(\mathcal{T}(z_0,\alpha_{\bar A}(v),v),Z^{z_0,\alpha_{\bar A}(v),v'}_{\mathcal{T}(z_0,\alpha_{\bar A}(v),v)},v\big)=\hat{\alpha}(v').
		\end{aligned}
	\end{equation}
	Consequently, $\hat{\alpha}\in\mathcal{A}_d$ and \eqref{eq_lem_value2} implies 
	\begin{equation}
		V^+(z_0)+\varepsilon\ge \inf_{\alpha\in\mathcal{A}_d}\sup_{v\in\mathcal{V}}J(z_0,\alpha(v),v).
	\end{equation}
	The desired inequality follows by passing $\varepsilon\to 0+$ on both sides of the last inequality above. The proof is complete.
\end{proof}
	Let us recall the following result regarding the existence of value for
infinite horizon two-person zero-sum differential games of complete information. Such games have been well-studied, see for example \cite{book_B&CD,E&S}.
\begin{mlem}[\cite{book_B&CD}]\label{lem_comp_info}
	Under Isaacs' condition \eqref{eq_IC1}, $V^+(z_0)=V^-(z_0)$, $\forall z_0\in\mathbb R^{n+1}$.	In addition, the value function $z\mapsto V(z)=V^\pm(z)$ is H\"{o}lder continuous independent of $(u,v)$ with exponent $\gamma$:
	$$\gamma=\left\{\begin{aligned}
	&1,&\text{if }L<\lambda,\\
	&\text{any }\gamma<1,&\text{if }L=\lambda,\\
	&\lambda/L,&\text{if }L>\lambda,
\end{aligned}\right.$$
where $L=\max(L_F,L_\ell)$.
\end{mlem}
\begin{mrmk}\label{rmk_comp_info}
	Fix $z\in{\R}^{n+1}$. For all $\varepsilon>0$, let $\alpha_z\in\mathcal{A}_d$ (resp. $\beta_z\in\mathcal{B}_d$) be an $\varepsilon$-optimal strategy for $V^+(z_0)$ (resp. $V^-(z_0)$). As a direct consequence of the above Lemma \ref{lem_comp_info}, there exists $\delta_z>0$ such that $\alpha_z$ (resp. $\beta_z$) is still a $2\varepsilon$-optimal strategy for $V^+(\zeta)$ (resp. $V^-(\zeta)$) for any $\zeta\in B(z;\delta_z)$.
\end{mrmk}

\section{Properties of the value functions $V^\pm$}
In this section, we aim to establish our first main result, i.e. the existence of value for game $\mathcal G(\mu_0)$ under Isaacs' condition \eqref{eq_IC2}. It has already been established in \cite{Wu_DGAA19} that, for the case $\mu_0$ is of finite support, game $\mathcal G(\mu_0)$ has a value under Isaacs' condition \eqref{eq_IC2}.
\begin{mprps}[\cite{Wu_DGAA19}]\label{prps_mainfin}
	Assuming Isaacs' condition \eqref{eq_IC2} and that $\mu_0$ is the finite combination of Dirac masses, namely $\mu_0=\sum_{i=1}^I q_i\delta_{z^i_0}\in\mathcal{P}({\R}^{n+1})$ for some $I\in{\N}_*$, then game $\mathcal G(\mu_0)$ has a value: $$V^+(\mu_0)=V^-(\mu_0).$$
\end{mprps}
\begin{mrmk}
	While the definition of signal-dependent NAD strategies slightly varies between this paper and \cite{Wu_DGAA19}, it follows from Proposition \ref{prps_AF} in the next subsection and the corresponding Lemma 3.2 from \cite{Wu_DGAA19} that the value functions in both papers are the same.
\end{mrmk}
\begin{mthm}\label{thm_main1}
	Under Isaacs' condition \eqref{eq_IC2}, for all $\mu_0\in\mathcal{P}({\R}^{n+1})$, game $\mathcal{G}(\mu_0)$ has a value $V(\mu_0)=V^\pm(\mu_0)$
\end{mthm}
\begin{proof}
	It is well-known that under Wasserstein distance $W_2$, any probability measure $\mu_0\in\mathcal P(\mathbb R^{n+1})$ can be approximated by a sequence of probability measures of finite support. Hence, the existence of value for game $\mathcal G(\mu_0)$ with $\mu_0\in\mathcal P(\mathbb R^{n+1})$ follows from Proposition \ref{prps_mainfin} and Proposition \ref{prps_regV1}. The proof is complete.
\end{proof}

\subsection{Alternative forms of $V^\pm$}
	To prove Theorem \ref{thm_main1}, we will need to obtain the continuity of $V^\pm$ with respect to $\mu_0\in\mathcal P(\mathbb R^{n+1})$. However, given SNAD strategies $(A,B)\in\mathcal A_s\times\mathcal B_s$, the map $z\mapsto (u_z,v_z)$ defined in Lemma \ref{SNAD_lem2} is not necessarily continuous. Therefore we can not deduce the regularity of $V^\pm$ directly from the uniform continuity of the cost.\par
	In order to overcome this obstacle, we write the value functions of game $\mathcal G(\mu_0)$ in alternative forms (cf. Proposition \ref{prps_AF}) which are in turn value functions of another differential game of symmetric incomplete information with both a running cost and a terminal cost at the controlled stopping time $\mathcal T(z_0,\cdot)$.
\begin{mprps}\label{prps_AF} Under Isaacs' condition \eqref{eq_IC1}, one has $\forall \mu_0\in\mathcal P(\mathbb R^{n+1})$,
	\begin{align*}
		V^+(\mu_0)&=\inf_{\alpha\in\mathcal A_d}\sup_{\beta\in\mathcal B_d}\int_{{\R}^{n+1}}\bar J(z,\alpha,\beta)d\mu_0(z),\\
		V^-(\mu_0)&=\sup_{\beta\in\mathcal B_d}\inf_{\alpha\in\mathcal A_d}\int_{{\R}^{n+1}}\bar J(z,\alpha,\beta)d\mu_0(z),
	\end{align*}
where $\bar J(z,\alpha,\beta)=\int^{\mathcal T(z,\alpha,\beta)}_0e^{-\lambda t} \ell(Z^{z,\alpha,\beta}_t,\alpha,\beta)dt
+e^{-\lambda {\mathcal T(z,\alpha,\beta)}}V^+(Z^{z,\alpha,\beta}_{\mathcal T(z,\alpha,\beta)}).$
\end{mprps}
\begin{proof}
	We only prove the first eqaution, since the second can be established symmetrically. Let us denote by $W^+(\mu_0):=\inf_{\alpha\in\mathcal A_d}\sup_{\beta\in\mathcal B_d}\int_{{\R}^{n+1}}\bar J(z,\alpha,\beta)d\mu_0(z)$ the right-hand side of the first equation.
	\paragraph*{Step 1: $V^+\le W^+$.}
	Let $\alpha_0$ be an $\varepsilon$-optimal strategy for $W^+(\mu_0)$, namely,
	\begin{align*}
		W^+(\mu)+\varepsilon\\
		\ge\sup_{\beta\in\mathcal B_d}\int_{{\R}^{n+1}}\Big[&\int^{\mathcal T(z,\alpha_0,\beta)}_0e^{-\lambda t} \ell(Z^{z,\alpha_0,\beta}_t,\alpha_0,\beta)dt
		+e^{-\lambda {\mathcal T(z,\alpha_0,\beta)}}V^+(Z^{z,\alpha_0,\beta}_{\mathcal T(z,\alpha_0,\beta)})\Big]d\mu_0(z).
	\end{align*}
	Since $\supp{\mu_0}$ is compact, it follows from \eqref{lip_hittime} and Assumptions \ref{Assum1} that there exists $\bar{\mathcal T}>0$ such that
	\begin{equation}
		0\le\mathcal T(z,u,v)\le \bar{\mathcal T},\ \forall z=(x,y)\in\supp{\mu_0}\txt{ and }(u,v)\in\mathcal U\times\mathcal V.
	\end{equation}
	Hence, for all $z\in\supp{\mu_0}$ and $(u,v)\in\mathcal U\times\mathcal V$, $\big|Z^{z,u,v}_{\mathcal T(z,u,v)}-z\big|\le \bar{\mathcal T}\|F\|_{\infty}$, and the set
	$Z_{\mu_0,\mathcal T}:=\big\{Z^{z,u,v}_{\mathcal T(z,u,v)}\ |\ z\in\supp{\mu_0},\ (u,v)\in\mathcal U\times\mathcal V\big\}$
	 is bounded. Let $M>0$ be sufficiently large so that $Z_{\mu,\mathcal T}\subset \bar B(\textbf{0};M)$. Let us choose for all $z\in\mathbb R^{n+1}$, $\alpha_z\in\mathcal A_d$ an $\varepsilon/2$-optimal strategy for $V^+(z)$, namely,
	 $$V^+(z)\ge \sup_{\beta\in\mathcal B_d}\int^\infty_0 e^{-\lambda t}\ell\big(Z^{z,\alpha_z,\beta}_t,\alpha_z,\beta\big)dt-\frac{\varepsilon}{2}.$$
	 In view of Remark \ref{rmk_comp_info}, there exists $\delta_z>0$ such that for all $z'\in B(z;\delta_z)$, $\alpha_z$ remains an $\varepsilon$-optimal strategy for $V^+(z')$. The family $\{B(z;\frac{\delta_z}{2})\}_{z\in \bar B(\textbf{0};M)}$ forms an open cover of $\bar B(\textbf{0};M)$, and thus there exists a finite cover $\bar B(\textbf{0};M)\subset \cup_{k=1}^N B\Big(z_k;\frac{\delta_{z_k}}{2}\Big)$.
	 Furthermore, from $\Big\{B\Big(z_k;\frac{\delta_{z_k}}{2}\Big)\Big\}_{1\le k\le N}$ we can construct a Borel partition of $\mathbb{R}^{n+1}$ by setting
	$$E_0=\emptyset,\ E_k=B\Big(z_k;\frac{\delta_{z_k}}{2}\Big)\backslash(\cup_{i=0}^{k-1}E_i), \forall 1\le k\le N\txt{ and }E_{N+1}=\mathbb{R}^{n+1}\backslash(\cup_{i=0}^{N}E_i).$$
	 Let $\delta:=\min_{1\le k\le N}\delta_{z_k}/2$. We denote by $\alpha_k$ the strategy $\alpha_{z_k}$ which is an $\varepsilon$-optimal strategy for $V^+(z)$ for all $z\in B(z_k;\delta_{z_k})\supset B\Big(z_k;\frac{\delta_{z_k}}{2}\Big)\supset E_k$. Let $\tau>0$ be a sufficiently small common delay of $\alpha_0,\alpha_1,...,\alpha_N$ such that $\tau\le \tau(\varepsilon,\delta)$ with
	$$\tau(\varepsilon,\delta)=\min\Big(\frac{\delta}{2(1+\|F\|_\infty)},\frac{\varepsilon}{2\|\ell\|_{\infty}},\frac{1}{2\lambda}\ln\Big(1-\frac{\varepsilon\lambda}{\|\ell\|_\infty+1}\Big),\frac{\varepsilon^{\frac{1}{\gamma}}}{2(\|F\|_{\infty}+1)C^{\frac{1}{\gamma}}}\Big).$$ 
	Here $\gamma>0$ and $C>0$ are defined as in Lemma \ref{lem_comp_info} such that $|V(z)-V(z)'|\le C\|z-z'\|^\gamma$ for all $z,z'\in\mathbb R^{n+1}$.	 
	 We construct a family of NAD strategies $\alpha_{m,k}\in\mathcal A_d$, for all $1\le k\le N$ and $1\le m\le \lceil \bar{\mathcal T}/\tau\rceil:=\inf\{m\in\mathbb{N}\ |\ \mathcal{T}/\tau\in (m-1,m]\}$ as follows. 
	\begin{equation*}
		\alpha_{m,k}(v)(t)=\left\{\begin{aligned}
			&\alpha_0(v)(t),&t\in[0,(m+1)\tau),\\
			&\alpha_k\big(v(\cdot+(m+1)\tau)\big)(t-(m+1)\tau),&t\ge (m+1)\tau.
		\end{aligned}\right.
	\end{equation*}
	One can check that $\alpha_{m,k}:\mathcal V\to\mathcal U$ is indeed an NAD strategy with delay $\tau>0$. We construct $\bar{A}\in\mathcal A_s$ a SNAD strategy for player 1 as follows: for all $(T,z,v)\in\R_+\times{\R}^{n+1}\times\mathcal{V}$,
	\begin{equation}
		\bar{A}(T,z,v)=\left\{\begin{aligned}
			&\alpha_{m,k}(v),&\text{if }z\in E_k\text{ and }T\in[(m-1)\tau,m\tau),\\
			&\alpha_0(v),&\text{else.}
		\end{aligned}\right.
	\end{equation}
	Since, for all $1\le k\le N$ and $1\le m\le \lceil \bar{\mathcal T}/\tau\rceil$, the maps $v\mapsto \alpha_{m,k}(v)$ and $v\mapsto \alpha_0(v)$ are Borel measurable, one deduce that $\bar{A}:\mathbb R_+\times\mathbb{R}^{n+1}\times\mathcal{V}\to\mathcal{U}$ is Borel measurable. In addition, the NAD property of $\bar{A}$ follows from the NAD property of $\alpha_0$ and $\{\alpha_{m,k}\}_{1\le k\le N,1\le m\le \lceil \bar{\mathcal T}/\tau\rceil}$, and we deduce that $\bar{A}\in\mathcal{A}_s$. 
	
	Fix arbitrary $z_0\in\supp{\mu_0}$ and $B\in\mathcal{B}_s$, and let $(u_{z_0},v_{z_0})\in\mathcal{U}\times\mathcal{V}$ denote the unique pair of controls associated to $(\bar{A},B)$ as defined in Lemma \ref{SNAD_lem2}. By the definition of $J(z_0,A,B)$,
	\begin{equation}\label{eq_prps_TSP1}
		\begin{aligned}J(z_0,A,B)=&\int^{\mathcal T(z_0,u_{z_0},v_{z_0})}_0e^{-\lambda t} \ell(Z^{z_0,u_{z_0},v_{z_0}}_t,u_{z_0},v_{z_0})dt\\
		&+\int_{\mathcal T(z_0,u_{z_0},v_{z_0})}^{+\infty}e^{-\lambda t} 	\ell(Z^{z_0,u_{z_0},v_{z_0}}_t,u_{z_0},v_{z_0})dt.\end{aligned}
	\end{equation}
	Let $\beta_B\in\mathcal B_d$ be the NAD strategy associated to $B$ defined as in Lemma \ref{SNAD_lem1}. Then by the above construction of $\bar{A}\in\mathcal{A}_s$, for almost every $t\in [0,\mathcal T(z_0,u_{z_0},v_{z_0})+\tau]$,
	\begin{align*}
		u_{z_0}(t)&=\bar{A}\big(\mathcal T(z_0,u_{z_0},v_{z_0}),Z^{z_0,u_{z_0},v_{z_0}}_{\mathcal T(z_0,u_{z_0},v_{z_0})},v_{z_0}\big)(t)=\alpha_0(v_{z_0})(t);\\
		v_{z_0}(t)&=B\big(\mathcal T(z_0,u_{z_0},v_{z_0}),Z^{z_0,u_{z_0},v_{z_0}}_{\mathcal T(z_0,u_{z_0},v_{z_0})},u_{z_0}\big)(t)=\beta_B(u_{z_0})(t).
	\end{align*}	
	Consequently, $\mathcal T(z_0,u_{z_0},v_{z_0})=\mathcal T(z_0,\alpha_0,\beta_B)$ by Lemma \ref{SNAD_lem1B}, and
	\begin{equation}\label{eq_prps_TSP2}
		\int^{\mathcal T(z_0,u_{z_0},v_{z_0})}_0e^{-\lambda t} \ell(Z^{z_0,u_{z_0},v_{z_0}}_t,u_{z_0},v_{z_0})dt=\int^{\mathcal T(z_0,\alpha_0,\beta_B)}_0e^{-\lambda t} \ell(Z^{z_0,\alpha_0,\beta_B}_t,\alpha_0,\beta_B)dt.
	\end{equation}
	Let us denote by $m_{z_0}$ the integer $m$ verifying $\mathcal T(z_0,u_{z_0},v_{z_0})\in [(m-1)\tau,m\tau)$, and we write
	\begin{equation}\label{eq_prps_TSP3}
		\begin{aligned}&\int_{\mathcal T(z_0,u_{z_0},v_{z_0})}^{+\infty}e^{-\lambda t} 	\ell(Z^{z_0,u_{z_0},v_{z_0}}_t,u_{z_0},v_{z_0})dt\\
		=&\int_{\mathcal T(z_0,u_{z_0},v_{z_0})}^{\tau (m_{z_0}+1)}e^{-\lambda t} 	\ell(Z^{z_0,u_{z_0},v_{z_0}}_t,u_{z_0},v_{z_0})dt+\int_{\tau (m_{z_0}+1)}^{+\infty}e^{-\lambda t} 	\ell(Z^{z_0,u_{z_0},v_{z_0}}_t,u_{z_0},v_{z_0})dt.\end{aligned}
	\end{equation}
	We have the following estimation for the first term on the right-hand side of \eqref{eq_prps_TSP3},
	\begin{equation}\label{eq_prps_TSP4}
		\Big|\int_{\mathcal T(z_0,u_{z_0},v_{z_0})}^{\tau (m_{z_0}+1)}e^{-\lambda t} 	\ell(Z^{z_0,u_{z_0},v_{z_0}}_t,u_{z_0},v_{z_0})dt\Big|\le \|\ell\|_\infty \big|\tau (m_{z_0}+1)-\mathcal T(z_0,u_{z_0},v_{z_0})\big|\le \varepsilon.
	\end{equation}
	It remains to estimate the other term. Assume that 
	$$Z^{z_0,\alpha_0,\beta_B}_{\mathcal T(z_0,\alpha_0,\beta_B)}\in E_k\txt{ for some }1\le k\le N.$$
	By the construction of the strategy $\bar{A}$, $u_{z_0}=a_{m_{z_0},k}(v_{z_0})$. Let us truncate the control $v_{z_0}$ by defining for all $m\ge 1$, $v^{m}_{z_0}=v_{z_0}(\cdot+m\tau)$. Hence
	$$u_{z_0}\big(\cdot+\tau (m_{z_0}+1)\big)=\alpha_{m_{z_0},k}(v_0)\big(\cdot+\tau (m_{z_0}+1)\big)=\alpha_k\Big(v^{m_{z_0}+1}_{z_0}\Big)\txt{ a.e. on }\mathbb{R}_+.$$
	Let us denote $Z^{z_0}_{B,\tau}:=Z^{z_0,u_{z_0},v_{z_0}}_{\tau(m_{z_0}+1)}$ for simplicity. It follows that, with $Z^{z_0,\alpha_0,\beta_B}_{\mathcal T(z_0,\alpha_0,\beta_B)}\in E_k$ and $\mathcal T(z_0,\alpha_0,\beta_B)\in [\tau(m_{z_0}-1),\tau m_{z_0})$,
	\begin{equation}\label{eq_prps_TSP5}
		\begin{aligned}
			&\int_{\tau (m_{z_0}+1)}^{+\infty}e^{-\lambda t} 	\ell(Z^{z_0,u_{z_0},v_{z_0}}_t,u_{z_0},v_{z_0})dt\\
			=&e^{-\lambda \tau (m_{z_0}+1)}\int_{0}^{+\infty}e^{-\lambda t} 	\ell(Z^{Z^{z_0}_{B,\tau},\alpha_k\Big(v^{m_{z_0}+1}_{z_0}\Big),v^{m_{z_0}+1}_{z_0}}_t,\alpha_k\Big(v^{m_{z_0}+1}_{z_0}\Big),v^{m_{z_0}+1}_{z_0})dt\\
			\le& e^{-\lambda \tau (m_{z_0}+1)}\sup_{v\in\mathcal V}\int_{0}^{+\infty}e^{-\lambda t} 	\ell(Z^{Z^{z_0}_{B,\tau},\alpha_k(v),v}_t,\alpha_k(v),v)dt.
		\end{aligned}
	\end{equation}		
	Recall that $\alpha_k$ is $\varepsilon$-optimal for $V^+(z)$, $\forall z\in B(z_k;\delta_{z_k})$. Since $Z^{z_0,\alpha_0,\beta_B}_{\mathcal T(z_0,\alpha_0,\beta_B)}\in E_k\subset B\big(z_k;\frac{\delta_{z_k}}{2}\big)$ and 
	$\|Z^{z_0}_{B,\tau}-Z^{z_0,\alpha_0,\beta_B}_{\mathcal T(z_0,\alpha_0,\beta_B)}\|\le 2\tau \|F\|_\infty\|\le \delta\le \frac{\delta_{z_k}}{2}$, the current state $Z^{z_0}_{B,\tau}$ at $t=\tau (m_{z_0}+1)$ stays in $B(z_k;\delta_{z_k})$ and consequently $\alpha_k$ remains $\varepsilon$-optimal for $V^+\big(Z^{z_0}_{B,\tau}\big)$. 	
	It follows from the above analysis and \eqref{eq_prps_TSP5} that one has
	\begin{equation}\label{eq_prps_TSP6}
		\begin{aligned}
			&\int_{\tau (m_{z_0}+1)}^{+\infty}e^{-\lambda t} 	\ell(Z^{z_0,u_{z_0},v_{z_0}}_t,u_{z_0},v_{z_0})dt
			\le e^{-\lambda \tau (m_{z_0}+1)} V^+(Z^{z_0}_{B,\tau})+\varepsilon\\
			\le& e^{-\lambda \mathcal T(z_0,\alpha_0,\beta_B)} V^+(Z^{z_0,\alpha_0,\beta_B}_{\mathcal T(z_0,\alpha_0,\beta_B)})+e^{-\lambda \mathcal T(z_0,\alpha_0,\beta_B)}|V^+(Z^{z_0,\alpha_0,\beta_B}_{\mathcal T(z_0,\alpha_0,\beta_B)})-V^+(Z^{z_0}_{B,\tau})|\\
			&+|V^+(Z^{z_0}_{B,\tau})|\big|e^{-\lambda \mathcal T(z_0,\alpha_0,\beta_B)}-e^{-\lambda \tau (m_{z_0}+1)}\big|+\varepsilon\\
			\le& e^{-\lambda \mathcal T(z_0,\alpha_0,\beta_B)} V^+(Z^{z_0,\alpha_0,\beta_B}_{\mathcal T(z_0,\alpha_0,\beta_B)})+C\big\|Z^{z_0,\alpha_0,\beta_B}_{\mathcal T(z_0,\alpha_0,\beta_B)}-Z^{z_0}_{B,\tau}\big\|^\gamma+\frac{\|\ell\|_{\infty}}{\lambda}|1-e^{-2\lambda \tau}|+\varepsilon\\
			\le& e^{-\lambda \mathcal T(z_0,\alpha_0,\beta_B)} V^+(Z^{z_0,\alpha_0,\beta_B}_{\mathcal T(z_0,\alpha_0,\beta_B)})+C(2\tau)^\gamma \|F\|_{\infty}^\gamma+\frac{\|\ell\|_{\infty}}{\lambda}|1-e^{-2\lambda \tau}|+\varepsilon\\
			\le& e^{-\lambda \mathcal T(z_0,\alpha_0,\beta_B)} V^+(Z^{z_0,\alpha_0,\beta_B}_{\mathcal T(y_0,\alpha_0,\beta_B)})+3\varepsilon.
		\end{aligned}
	\end{equation}
	In combining \eqref{eq_prps_TSP1}-\eqref{eq_prps_TSP6}, we have, for any $B\in\mathcal{B}_S$, $\mu_0$-almost surely
	\begin{equation}\label{eq_prps_TSP7}
		\begin{aligned}		
			&J(z_0,\bar{A},B)-4\varepsilon\\
			\le&\int^{\mathcal T(z_0,\alpha_0,\beta_B)}_0e^{-\lambda t} \ell(Z^{z_0,\alpha_0,\beta_B}_t,\alpha_0,\beta_B)dt+e^{-\lambda \mathcal T(z_0,\alpha_0,\beta_B)} V^+(Z^{z_0,\alpha_0,\beta_B}_{\mathcal T(z_0,\alpha_0,\beta_B)}).
		\end{aligned}
	\end{equation}
	It follows that for all $B\in\mathcal{B}_s$,
	\begin{equation}
		\begin{aligned}
			\int_{{\R}^{n+1}}J(z,\bar{A},B)d\mu_0(z)-5\varepsilon&\le \int_{{\R}^{n+1}}\bar J(z,\alpha_0,\beta_B)d\mu_0(z)-\varepsilon\\
			&\le \sup_{\beta\in\mathcal B_d}\int_{{\R}^{n+1}}\bar J(z,\alpha_0,\beta)d\mu_0(z)-\varepsilon\le W^+(\mu_0).
		\end{aligned}
	\end{equation}
	Consequently, taking the supremum over $B\in\mathcal{B}_s$ on both sides of the last inequality above, we obtain
	\begin{equation}\label{eq_prps_TSP8}
		V^+(\mu_0)-5\varepsilon\le \sup_{B\in\mathcal B_s}\int_{{\R}^{n+1}}J(z,\bar{A},B)d\mu_0(z)-5\varepsilon\le W^+(\mu_0).
	\end{equation}
	Passing $\varepsilon\to 0+$ on both sides of \eqref{eq_prps_TSP8} leads to the desired inequality $V^+(\mu_0)\le W^+(\mu_0)$.
	\paragraph*{Step 2: $V^+\ge W^+$.} To prove the opposite inequality, let us recall that by Lemma \ref{lem_comp_info}, under Isaacs' condition \eqref{eq_IC1}, $V^+(z)=V^-(z)$ for all $z\in\mathbb R^{n+1}$. Let $A^*\in\mathcal A_s$ be an $\varepsilon$-optimal strategy for $V^+(\mu_0)$, namely,
	\begin{equation}\label{eq_prps_TSP9}
		V^+(\mu_0)\ge \sup_{B\in\mathcal B_s}\int_{{\R}^{n+1}}J(z,A^*,B)d\mu_0(z)-\varepsilon,
	\end{equation}
	and we denote by $\alpha^*\in\mathcal A_d$ the NAD strategy associated to $A^*$ as defined in Lemma \ref{SNAD_lem1}.\par
	As in Step 1, there exists a finite collection of open balls $\big(B\big(z'_k;\frac{\delta'_k}{2}\big)\big)_{1\le k\le N'}$ and a finite family of NAD strategies for player 2: $(\beta_k)_{1\le k\le N'}$ such that:
	\begin{itemize}
		\item
			 $\big(B\big(z'_k;\frac{\delta'_k}{2}\big)\big)_{1\le k\le N'}$ forms a finite open cover of $$\bar B(\textbf{0};M)\supset Z_{\mu_0,\mathcal T}=\big\{Z^{z_0,u,v}_{\mathcal T(z_0,u,v)}\ |\ z_0\in\supp{\mu_0},\ (u,v)\in\mathcal{U}\times\mathcal{V}\big\};$$
		\item
			 for all $z\in B(z'_k;\delta'_k)$, $\beta_k$ is an $\varepsilon$-optimal strategy for $V^-(z)$, i.e.
			 $$V^-(z)-\varepsilon\le \inf_{u\in\mathcal{U}} \int^\infty_0 e^{-\lambda t}\ell\big(Z^{z,u,\beta_k(u)}_t,u,\beta_k(u)\big)dt.$$
	\end{itemize}
	We construct a Borel partition of $\mathbb{R}^{n+1}$ by setting
	$$E'_0=\emptyset,\ E'_k=B\Big(z'_k;\frac{\delta'_{k}}{2}\Big)\backslash(\cup_{i=0}^{k-1}E'_i), \forall 1\le k\le N'\txt{ and }E'_{N'+1}=\mathbb R^{n+1}\backslash(\cup_{i=0}^{N'}E'_i).$$
	and we denote $\delta':=\min_{1\le k\le N'}\delta'_{k}/2$. Similar to the process in Step 1, let $\beta_0\in\mathcal B_d$ be an arbitrary NAD strategy of player 2. We mimic the process in Step 1 to construct a SNAD strategy $\bar{B}\in\mathcal B_s$ from $\beta_0$ and the collection of $\varepsilon$-optimal strategies $\{\beta_k\}_{1\le k\le N'}$.
	
	Let $\tau'>0$ be a common delay of $\{\beta_k\}_{0\le k\le N'}$ such that $\tau'\le \tau(\varepsilon,\delta')$. We define, for all $1\le k\le N'$ and $1\le m\le \lceil \bar{\mathcal T}/\tau'\rceil$,  the NAD strategy $\beta_{m,k}\in\mathcal B_d$ as follows. 
	\begin{equation*}
		\forall u\in\mathcal U,\ \beta_{m,k}(u)(t)=\left\{\begin{aligned}
			&\beta_0(u)(t),&t\in[0,(m+1)\tau'),\\
			&\beta_k\big(u(\cdot+(m+1)\tau')\big)(t-(m+1)\tau'),&t\ge (m+1)\tau'.
		\end{aligned}\right.
	\end{equation*}
	One can verify that $\beta_{m,k}:\mathcal U\to\mathcal V$ is an NAD strategy with delay $\tau'>0$. We proceed to construct $\bar{B}\in\mathcal{B}_s$ follows: for all $(T,z,u)\in\R_+\times{\R}^{n+1}\times\mathcal{U}$,
	\begin{equation}
		\bar{B}(T,z,u)=\left\{\begin{aligned}
			&\beta_{m,k}(u),&\text{if }z\in E'_k\text{ and }T\in[(m-1)\tau',m\tau'),\\
			&\beta_0(u),&\text{else.}
		\end{aligned}\right.
	\end{equation}
	Similar to the strategy $\bar{A}\in\mathcal A_s$ constructed in Step 1, one can check that $\bar{B}\in\mathcal B_s$.
	
	Fix $z_0\in\supp{\mu_0}$. Let $(u'_{z_0},v'_{z_0})$ denote the unique pair of admissible controls associated to $(A^*,\bar{B})$ defined as in Lemma \ref{SNAD_lem2}. By Lemma \ref{SNAD_lem1} and the above construction of $\bar{B}$, for almost every $t\in [0,\mathcal T(z_0,u'_{z_0},v'_{z_0})+\tau']$,
	\begin{align*}
		u'_{z_0}(t)&=A^*\big(\mathcal T(z_0,u'_{z_0},v'_{z_0}),Z^{z_0,u'_{z_0},v'_{z_0}}_{\mathcal T(z_0,u'_{z_0},v'_{z_0})},v'_{z_0}\big)(t)=\alpha^*(v'_{z_0})(t);\\
		v'_{z_0}(t)&=\bar{B}\big(\mathcal T(z_0,u'_{z_0},v'_{z_0}),Z^{z_0,u'_{z_0},v'_{z_0}}_{\mathcal T(z_0,u'_{z_0},v'_{z_0})},u'_{z_0}\big)(t)=\beta_0(u'_{z_0})(t).
	\end{align*}	
	Hence $\mathcal T(z_0,u'_{z_0},v'_{z_0})=\mathcal T(z_0,\alpha^*,\beta_0)$ by Lemma \ref{SNAD_lem1B}, and it follows that
	\begin{equation}\label{eq_prps_TSP10}
		\begin{aligned}J(z_0,A^*,\bar{B})=&\int^{\mathcal T(z_0,\alpha^*,\beta_0)}_0e^{-\lambda t} \ell(Z^{z_0,\alpha^*,\beta_0}_t,\alpha^*,\beta_0)dt\\
		&+\int_{\mathcal T(z_0,\alpha^*,\beta_0)}^{+\infty}e^{-\lambda t} 	\ell(Z^{z_0,u'_{z_0},v'_{z_0}}_t,u'_{z_0},v'_{z_0})dt.\end{aligned}
	\end{equation}
	By denoting $m'_{z_0}$ the integer $m$ such that $\mathcal T(z_0,u'_{z_0},v'_{z_0})\in [(m-1)\tau',m\tau')$, the last term on the right-hand side of \eqref{eq_prps_TSP10} can be written as:
	\begin{equation}\label{eq_prps_TSP11}
		\begin{aligned}&\int_{\mathcal T(z_0,\alpha^*,\beta_0)}^{+\infty}e^{-\lambda t} 	\ell(Z^{z_0,u'_{z_0},v'_{z_0}}_t,u'_{z_0},v'_{z_0})dt\\
		=&\int_{\mathcal T(z_0,\alpha^*,\beta_0)}^{\tau' (m'_{z_0}+1)}e^{-\lambda t} 	\ell(Z^{z_0,u'_{z_0},v'_{z_0}}_t,u'_{z_0},v'_{z_0})dt+\int_{\tau' (m'_{z_0}+1)}^{+\infty}e^{-\lambda t} 	\ell(Z^{z_0,u'_{z_0},v'_{z_0}}_t,u'_{z_0},v'_{z_0})dt.\end{aligned}
	\end{equation}
	Similar to \eqref{eq_prps_TSP4}, we have
	\begin{equation}\label{eq_prps_TSP12}
		\Big|\int_{\mathcal T(z_0,\alpha^*,\beta_0)}^{\tau' (m'_{z_0}+1)}e^{-\lambda t} 	\ell(Z^{z_0,u'_{z_0},v'_{z_0}}_t,u'_{z_0},v'_{z_0})dt\Big|\le \|\ell\|_\infty \big|\tau' (m'_{z_0}+1)-\mathcal T(z_0,\alpha^*,\beta_0)\big|\le \varepsilon.
	\end{equation}
	It remains thus to estimate the second term on the right-hand side of \eqref{eq_prps_TSP11}. Assume that $$Z^{z_0,\alpha^*,\beta_0}_{\mathcal T(z_0,\alpha^*,\beta_0)}\in E'_k\txt{ for some }1\le k\le N'.$$
	By the construction of strategy $\bar{B}$, $v'_{z_0}=\beta_{m'_{z_0},k}(u'_{z_0})$. In view of the definition of $\beta_{m'_{z_0},k}$, let us further truncate the control $u'_{z_0}$ by defining for all $m\ge 1$, ${u'}^{m}_{z_0}=u'_{z_0}(\cdot+m\tau')$. Hence
	$$v'_{z_0}\big(\cdot+\tau' (m'_{z_0}+1)\big)=\beta_{m'_{z_0},k}(u'_{z_0})\big(\cdot+\tau' (m'_{z_0}+1)\big)=\beta_k\Big({u'}^{m'_{z_0}+1}_{z_0}\Big)\txt{ a.e. on }\mathbb{R}_+.$$
	Let us write $Z^{z_0}_{\beta_0,\tau'}:=Z^{z_0,u'_{z_0},v'_{z_0}}_{\tau'(m'_{z_0}+1)}$ for simplicity. It follows that, with $Z^{z_0,\alpha_0,\beta_B}_{\mathcal T(z_0,\alpha^*,\beta_0)}\in E'_k$ and $\mathcal T(z_0,\alpha^*,\beta_0)\in [\tau'(m'_{z_0}-1),\tau' m'_{z_0})$,
	\begin{equation}\label{eq_prps_TSP13}
		\begin{aligned}
			&\int_{\tau' (m'_{z_0}+1)}^{+\infty}e^{-\lambda t} 	\ell(Z^{z_0,u'_{z_0},v'_{z_0}}_t,u'_{z_0},v'_{z_0})dt\\
			=&e^{-\lambda \tau' (m'_{z_0}+1)}\int_{0}^{+\infty}e^{-\lambda t} 	\ell(Z^{Z^{z_0}_{\beta_0,\tau'},{u'}^{m'_{z_0}+1}_{z_0},\beta_k\Big({u'}^{m'_{z_0}+1}_{z_0}\Big)}_t,{u'}^{m'_{z_0}+1}_{z_0},\beta_k\Big({u'}^{m'_{z_0}+1}_{z_0}\Big))dt\\
			\ge& e^{-\lambda \tau' (m'_{z_0}+1)}\inf_{u\in\mathcal U}\int_{0}^{+\infty}e^{-\lambda t} 	\ell(Z^{Z^{z_0}_{\beta_0,\tau'},u,\beta_k(u)}_t,u,\beta_k(u))dt.
		\end{aligned}
	\end{equation}		
	But $\beta_k$ is $\varepsilon$-optimal for $V^-(z)$, $\forall z\in B(z'_k;\delta'_{k})$. Since $Z^{z_0,\alpha^*,\beta_0}_{\mathcal T(z_0,\alpha^*,\beta_0)}\in E_k\subset B\big(z'_k;\frac{\delta'_{k}}{2}\big)$ and 
	$$\|Z^{z_0}_{\beta_0,\tau'}-Z^{z_0,\alpha^*,\beta_0}_{\mathcal T(z_0,\alpha^*,\beta_0)}\|\le 2\tau' \|F\|_\infty\le \delta'\le \frac{\delta'_{k}}{2},$$ 
	the state $Z^{z_0}_{\beta_0,\tau'}$ stays in $B(z'_k;\delta'_{k})$ and thus $\beta_k$ is $\varepsilon$-optimal for $V^-\big(Z^{z_0}_{\beta_0,\tau'}\big)$. 	
	Arguing as in \eqref{eq_prps_TSP6}, we obtain the following estimation:
	\begin{equation}\label{eq_prps_TSP14}
		\begin{aligned}
			&\int_{\tau' (m'_{z_0}+1)}^{+\infty}e^{-\lambda t} 	\ell(Z^{z_0,u'_{z_0},v'_{z_0}}_t,u'_{z_0},v'_{z_0})dt
			\ge e^{-\lambda \tau' (m'_{z_0}+1)} V^-(Z^{z_0}_{\beta_0,\tau'})-\varepsilon\\
			\ge &e^{-\lambda \mathcal T(z_0,\alpha^*,\beta_0)} V^-(Z^{z_0,\alpha^*,\beta_0}_{\mathcal T(z_0,\alpha^*,\beta_0)})-3\varepsilon
			=e^{-\lambda \mathcal T(z_0,\alpha^*,\beta_0)} V^+(Z^{z_0,\alpha^*,\beta_0}_{\mathcal T(z_0,\alpha^*,\beta_0)})-3\varepsilon.
		\end{aligned}
	\end{equation}
	Since $z_0\in\supp{\mu_0}$ is arbitrary, in combining \eqref{eq_prps_TSP10}-\eqref{eq_prps_TSP14}, we have $\mu_0$-almost surely
	\begin{equation}\label{eq_prps_TSP15}
		\begin{aligned}		
			&J(z_0,A^*,\bar{B})+4\varepsilon\\
			\ge&\int^{\mathcal T(z_0,\alpha^*,\beta_0)}_0e^{-\lambda t} \ell(Z^{z_0,\alpha^*,\beta_0}_t,\alpha^*,\beta_0)dt+e^{-\lambda \mathcal T(z_0,\alpha^*,\beta_0)} V^+(Z^{z_0,\alpha^*,\beta_0}_{\mathcal T(z_0,\alpha^*,\beta_0)})\\
			=&\bar J(z_0,\alpha^*,\beta_0),
		\end{aligned}
	\end{equation}
	The above inequality further implies
	\begin{equation}\label{eq_prps_TSP16}
		\begin{aligned}
			V^+(\mu_0)\ge\sup_{B\in\mathcal B_s}\int_{{\R}^{n+1}}J(z,A^*,B)d\mu_0(z)-\varepsilon\ge& \int_{{\R}^{n+1}}J(z,A^*,\bar{B})d\mu_0(z)-\varepsilon\\
			\ge& \int_{{\R}^{n+1}}\bar J(z,\alpha^*,\beta_0)d\mu_0(z)-5\varepsilon.
		\end{aligned}
	\end{equation}
	However, \eqref{eq_prps_TSP16} holds for all $\beta_0\in\mathcal B_d$, therefore
	\begin{equation}\label{eq_prps_TSP17}
		V^+(\mu_0)\ge \sup_{\beta\in\mathcal B_d}\int_{{\R}^{n+1}}\bar J(z,\alpha^*,\beta)d\mu_0(z)-5\varepsilon\ge W^+(\mu_0)-5\varepsilon.
	\end{equation}
	The desired inequality follows by passing $\varepsilon\to 0+$ on both sides of the above inequality. The proof is complete.
\end{proof}
Before proceeding to prove the regularity of the value functions, let us state the following corollary which can be proved following similar arguments as in the proof of Lemma \ref{lem_value1}.
\begin{mcor}\label{cor_AF}
	Under Isaacs' condition \eqref{eq_IC1}, one has for all $ \mu_0\in\mathcal P(\mathbb R^{n+1})$,
	\begin{align*}
		V^+(\mu_0)&=\inf_{\alpha\in\mathcal A_d}\sup_{v\in\mathcal V}\int_{{\R}^{n+1}}\bar J(z,\alpha(v),v)d\mu_0(z);\\
		V^-(\mu_0)&=\sup_{\beta\in\mathcal B_d}\inf_{u\in\mathcal U}\int_{{\R}^{n+1}}\bar J(z,u,\beta(u))d\mu_0(z).
	\end{align*}
\end{mcor}
\subsection{Regularity of $V^\pm$}
The main benefit of writing the value functions in their alternative forms from Proposition \ref{prps_AF} is that the cost $z\mapsto\bar J(z,\alpha,\beta)$ is uniformly continuous independently of $(\alpha,\beta)\in\mathcal A_d\times\mathcal B_d$, as stated in the following lemma which can be derived from standard estimates.
\begin{mlem}\label{lem_reg1}
	$\bar J(\cdot,\alpha,\beta)$ is uniformly continuous independently of $(\alpha,\beta)$. More precisly, there exists $C>0$ such that for all $(\alpha,\beta)\in\mathcal A_d\times\mathcal B_d$ and $z,z'\in\mathbb{R}^{n+1}$, one has
	\begin{equation}\label{eq_lem_reg1}
		\big|\bar J(z,\alpha,\beta)-\bar J(z',\alpha,\beta)\big|\le C(\|z-z'\|+\|z-z'\|^\gamma)
	\end{equation}
	where $0<\gamma\le 1$ is given by Lemma \ref{lem_costcont}.
\end{mlem}

To obtain the continuity of $V^\pm$ with respect to $\mu_0\in\mathcal P(\mathbb R^{n+1})$, we recall the useful technical result below.
\begin{mlem}\label{lem_regtech}
	Let $\mathbb A$, $\mathbb B$ be arbitrary sets and $f_1,\ f_2$ real valued maps defined on $\mathbb{A}\times \mathbb{B}$ such that for some constant $C>0$, one has
	$$\sup_{(a,b)\in\mathbb A\times\mathbb B}|f_1(a,b)-f_2(a,b)|\le C.$$
	Then $|\inf_{a\in\mathbb A}\sup_{b\in\mathbb B}f_1(a,b)-\inf_{a\in\mathbb A}\sup_{b\in\mathbb B}f_2(a,b)|\le C$.
\end{mlem}
\begin{mprps}\label{prps_regV1}
	Both $V^\pm$ are bounded, and uniformly continuous with respect to the Wasserstein distance $W_2$. More precisely, there exists a constant $C>0$ such that, for any $\mu_1,\mu_2\in\mathcal P(\mathbb R^{n+1})$,
	\begin{equation}
		\big|V^\pm(\mu_1)-V^\pm(\mu_2)\big|\le C(W_2(\mu_1,\mu_2)+W_2(\mu_1,\mu_2)^\gamma)
	\end{equation}
	where $0<\gamma\le 1$ is given by Lemma \ref{lem_costcont}.
\end{mprps}
\begin{proof}
	The boundedness of $V^\pm$ follows from that of the cost function. 
	
	Fix $\mu_1,\mu_2\in\mathcal P(\mathbb R^{n+1})$ and let $\pi\in\Pi(\mu_1,\mu_2)$ be an optimal transport plan for $W_2(\mu_1,\mu_2)$. For any $(\alpha,\beta)\in\mathcal A_d\times\mathcal B_d$, one has
	\begin{align*}
		&\Big|\int_{{\R}^{n+1}}\bar J(z,\alpha,\beta)d\mu_1(z)-\int_{{\R}^{n+1}}\bar J(z',\alpha,\beta)d\mu_2(z')\Big|\\
		\le &\int_{{\R}^{n+1}\times{\R}^{n+1}}\big|\bar J(z,\alpha,\beta)-\bar J(z',\alpha,\beta)\big|d\pi(z,z').
	\end{align*}
	It follows further from \eqref{eq_lem_reg1} and H\"{o}lder's inequality that,
	\begin{align*}
		&\Big|\int_{{\R}^{n+1}}\bar J(z,\alpha,\beta)d\mu_1(z)-\int_{{\R}^{n+1}}\bar J(z',\alpha,\beta)d\mu_2(z')\Big|\\
		\le& C\Big(\int_{{\R}^{2n+2}}\|z-z'\|d\pi(z,z') + \int_{{\R}^{2n+2}}\|z-z'\|^{\gamma}d\pi(z,z')\Big)\\
		\le& C\big(W_2(\mu_1,\mu_2)+W_2(\mu_1,\mu_2)^\gamma\big)
	\end{align*}
	Since $(\alpha,\beta)\in\mathcal A_d\times\mathcal B_d$ is arbitrary, we deduce from the above inequality that,
	\begin{equation}\label{eq_prps_regV1A}
		\begin{aligned}
			&\sup_{(\alpha,\beta)\in\mathcal A_d\times\mathcal B_d}\Big|\int_{{\R}^{n+1}}\bar J(z,\alpha,\beta)d\mu_1(z)-\int_{{\R}^{n+1}}\bar J(z',\alpha,\beta)d\mu_2(z')\Big|\\
			\le &C\big(W_2(\mu_1,\mu_2)+W_2(\mu_1,\mu_2)^\gamma\big).
		\end{aligned}
	\end{equation}
	Applying Lemma \ref{lem_regtech} and Proposition \ref{prps_AF}, we obtain both
	\begin{equation}
		\big|V^+(\mu_1)-V^+(\mu_2)\big|\le C(W_2(\mu_1,\mu_2)+W_2(\mu_1,\mu_2)^\gamma),
	\end{equation}
	and 
	\begin{equation}
		\big|V^-(\mu_1)-V^-(\mu_2)\big|\le C(W_2(\mu_1,\mu_2)+W_2(\mu_1,\mu_2)^\gamma).
	\end{equation}
	The proof is complete.
\end{proof}

\section{Extended Value Functions}
In this section we define, for our game $\mathcal{G}(\mu_0)$, the extended value functions similar to those investigated in \cite{J&Q18}, and we study their properties. Such an extension allows us to ``separate" the initial probability measure $\mu_0$ and the evolution of its support as the game progresses, which in turn enables us to obtain dynamic programming principles for the extended values and to write the appropriate Hamilton-Jacobi-Isaacs equation associated with our problem. Recall that $Z=X\times [0,M_0+\eta]\subset{\R}^{n+1}$ where $X$ is a compact subset of ${\R}^n$ and $\eta>0$, and the set of probability measures on $Z$ is denoted by $\Delta(Z)$. Throughout the rest of this paper, in addition to Assumptions \ref{Assum1} and Isaacs' condition \eqref{eq_IC3}, we assume furthermore the following:
\begin{massum}\label{Assum2}
	\begin{enumerate}[label=\roman*)]
		\item $Z$ is an invariant set for \eqref{dym_sys2}, namely 
			$$\forall(t,z,u,v)\in {\R}_+\times Z\times \mathcal{U}\times\mathcal{V},\ Z^{z,u,v}_t\in Z;$$
		\item
			The map $F:{\R}^{n+1}\times\U\times\V\to{\R}^{n+1}$ is Lipschitz continuous in all variables with Lipschitz constant $L_F>0$;
		\item
			$\lambda>L:=\max(L_F,L_\ell)$.
	\end{enumerate}
\end{massum}
	As a direct consequence of Assumptions \ref{Assum2} and Lemma \ref{lem_reg1}, we have the following
\begin{mcor}\label{cor_reg1}
	Under Assumptions \ref{Assum2}, $z\mapsto \bar J(z,\alpha,\beta)$ is Lipschitz continuous uniformly with repect to $(\alpha,\beta)$. More precisly, there exists $C>0$ such that for all $(\alpha,\beta)\in\mathcal A_d\times\mathcal B_d$ and $z,z'\in\mathbb{R}^{n+1}$, one has
	\begin{equation}\label{eq_cor_reg1}
		\big|\bar J(z,\alpha,\beta)-\bar J(z',\alpha,\beta)\big|\le C\|z-z'\|.
	\end{equation}
\end{mcor}

Recall that for any $(A,B)\in\mathcal A_s\times\mathcal B_s$ and any $\Phi\in L^2_{\mu_0}(Z,Z)$, the map 
	$$(A,B):\mathbb R^{n+1}\to\mathcal U\times\mathcal V:z\mapsto(u_{\Phi(z)},v_{\Phi(z)})$$ 
	is measurable with $(u_{z},v_{z})$ defined as in Lemma \ref{SNAD_lem2}. Let us introduce the following notion of extended value functions.
	
\begin{mdef}
	Let $\mu_0 \in\Delta(Z)$, $\Phi\in L^2_{\mu_0}(Z,Z)$, we define:
	\begin{align*}
		\mathcal V^+(\Phi,\mu_0)=\inf_{A\in\mathcal A_s}\sup_{B\in\mathcal B_s}\int_{Z}d\mu_0(z)\Big[\int_0^\infty e^{-\lambda t}\ell\big(Z^{\Phi(z),(A,B)(z)}_t,(A,B)(z)(t)\big)dt\Big],\\
		\mathcal V^-(\Phi,\mu_0)=\sup_{B\in\mathcal B_s}\inf_{A\in\mathcal A_s}\int_{Z}d\mu_0(z)\Big[\int_0^\infty e^{-\lambda t}\ell\big(Z^{\Phi(z),(A,B)(z)}_t,(A,B)(z)(t)\big)dt\Big].
	\end{align*}
\end{mdef}

	To simplify notations, for any $(A,B)\in\mathcal A_s\times\mathcal B_s$ and for any $(\alpha,\beta)\in\mathcal A_d\times\mathcal B_d$, we write
	\begin{align*}
		\mathcal J(\Phi,\mu_0,A,B)&:=\int_{Z}J(\Phi(z),u_{\Phi(z)},v_{\Phi(z)})d\mu_0(z),\\
		\bar{\mathcal J}(\Phi,\mu_0,\alpha,\beta)&:=\int_{Z}\bar J(\Phi(z),u_{\alpha\beta},\beta_{\alpha\beta})d\mu_0(z),
	\end{align*}
	where the map $z\mapsto(u_z,v_z)$ and the pair $(u_{\alpha\beta},v_{\alpha\beta})$ are respectively defined as in Lemma \ref{SNAD_lem2} and Lemma \ref{SNAD_lem1B}. In this regard, the extended value functions can be rewritten as
	\begin{align*}
		\mathcal V^+(\Phi,\mu_0)=\inf_{A\in\mathcal A_s}\sup_{B\in\mathcal B_s}\mathcal J(\Phi,\mu_0,A,B),\ 
		\mathcal V^-(\Phi,\mu_0)=\sup_{B\in\mathcal B_s}\inf_{A\in\mathcal A_s}\mathcal J(\Phi,\mu_0,A,B).
	\end{align*}

\begin{mrmk}
	The definitions of the extended values in this paper are slightly different from those in \cite{J&Q18}. In our case, the controls generated by SNAD strategies depend on the unknown initial position through signal revelation during game play, while in \cite{J&Q18}, the controls depend on private signals communicated before the game commences.
\end{mrmk}

\begin{mlem}\label{lem_extvalue}
For all $\mu_0\in\Delta(Z)$ and $\Phi\in L^2_{\mu_0}(Z,Z)$, one has
\begin{enumerate}[label=\roman*)]
	\item
		$\mathcal V^\pm(Id,\mu_0)=V^\pm(\mu_0)$;
	\item
		$\mathcal V^\pm(\Phi,\mu_0)= V^\pm(\Phi\sharp\mu_0)$;
	\item
		the reformation of $\mathcal V^\pm$:
		\begin{align*}
			\mathcal V^+	(\Phi,\mu_0)=&\inf_{\alpha\in\mathcal A_d}	\sup_{\beta\in\mathcal{B}_d}\bar{\mathcal J}(\Phi,\mu_0,\alpha,\beta)=\inf_{\alpha\in\mathcal A_d}	\sup_{v\in\mathcal V}\bar{\mathcal J}(\Phi,\mu_0,\alpha(v),v),\\
			\mathcal V^-(\Phi,\mu_0)=&\sup_{\beta\in\mathcal B_d}\inf_{\alpha\in\mathcal{A}_d}\bar{\mathcal J}(\Phi,\mu_0,\alpha,\beta)=\sup_{\beta\in\mathcal B_d}\inf_{u\in\mathcal U}\bar{\mathcal J}(\Phi,\mu_0,u,\beta(u)).
	\end{align*}
\end{enumerate}
\end{mlem}
\begin{proof}
	i) is a direct consequence of the definition of $\mathcal V^\pm$, and ii) follows from the change of variable formula since, given any pair of SNAD strategies $(A,B)$,
	\begin{align*}
		\mathcal J(\Phi,\mu_0,A,B)=\int_{Z}J(\Phi(z),u_{\Phi(z)},v_{\Phi(z)})d\mu_0(z)=\int_{Z}J(z,u_{z},v_{z})d\Phi\sharp\mu_0(z).
	\end{align*}
	Therefore 
	\begin{align*}
		\mathcal V^+(\Phi,\mu_0)&=\inf_{A\in\mathcal A_s}\sup_{B\in\mathcal B_s}\mathcal J(\Phi,\mu_0,A,B)\\
		&=\inf_{A\in\mathcal A_s}\sup_{B\in\mathcal B_s}\int_{Z}J(z,A,B)d\Phi\sharp\mu_0(z)=V^+(\Phi\sharp\mu_0).
	\end{align*}
	Similarly, we can prove that $\mathcal V^-(\Phi,\mu_0)=V^-(\Phi\sharp\mu_0)$. Finally, iii) follows from ii) and Corollary \ref{cor_AF}. The proof is complete.
\end{proof}

\subsection{Regularity of the Extended Value Functions}
In this subsection, we state and prove several lemmas which establish the continuity of the extended value functions $\mathcal V^\pm$ and provide us some useful estimates.
\begin{mlem}\label{lem_reg2}
	Under Assumptions \ref{Assum2}, there exists $C>0$ such that $\forall \mu_0\in \Delta(Z)$ and $\forall\Phi,\Psi\in L^2_{\mu_0}(Z,Z)$,
	\begin{equation*}
		|\mathcal V^\pm(\Phi,\mu_0)-\mathcal V^\pm(\Psi,\mu_0)|\le C\int_{z}\big\|\Phi(z)-\Psi(z)\big\|d\mu_0(z).
	\end{equation*}
\end{mlem}
\begin{proof}
	By Lemma \ref{lem_regtech} and Lemma \ref{lem_extvalue}, $\forall \mu_0\in \Delta(Z)$ and $\forall\Phi,\Psi\in L^2_{\mu_0}(Z,Z)$,
	\begin{equation}
		|\mathcal V^\pm(\Phi,\mu_0)-\mathcal V^\pm(\Psi,\mu_0)|
		\le \sup_{\alpha\in\mathcal{A}_d,\beta\in\mathcal{B}_d}\int_{Z}\big|\bar J(\Phi(z),\alpha,\beta)-\bar J(\Psi(z),\alpha,\beta)\big|d\mu_0(z).
	\end{equation}
	It follows from Corollary \ref{cor_reg1} that there exists $C>0$ such that
	\begin{equation}
		\begin{aligned}
			|\mathcal V^\pm(\Phi,\mu_0)-\mathcal V^\pm(\Psi,\mu_0)|
			\le &C\int_{Z}\|\Phi(z)-\Psi(z)\|d\mu_0(z).
		\end{aligned}
	\end{equation}
	The proof is complete.
\end{proof}
An immediate consequence of the above lemma is the following:
\begin{mcor}\label{cor_reg2}
	Under Assumptions \ref{Assum2}, for all $\mu_0\in\Delta(Z)$, both maps $\Phi\mapsto \mathcal V^{\pm}(\Phi,\mu_0)$ are bounded and Lipschitz continuous from $C(Z,Z)$ to $\R$ with respect to the infinity norm $\|\cdot\|_{\infty}$, i.e. there exists some constant $C>0$ such that for all $\Phi,\Psi\in C(Z,Z)$,
	\begin{equation*}
		|\mathcal V^\pm(\Phi,\mu_0)-\mathcal V^\pm(\Psi,\mu_0)|\le C\big\|\Phi-\Psi\big\|_\infty.
	\end{equation*}
\end{mcor}
\begin{mlem}\label{lem_reg3}
	Under Assumptions \ref{Assum2}, there exists $C>0$ such that for all Lipschitz continuous map $\Phi:Z\to Z$, and for any $\mu_1,\mu_2\in \Delta(Z)$ 
	\begin{equation*}
		|\mathcal V^\pm(\Phi,\mu_1)-\mathcal V^\pm(\Phi,\mu_2)|\le Lip(\Phi)C W_2(\mu_1,\mu_2),
	\end{equation*}
	where $Lip(\Phi)>0$ denotes the Lipschitz constant of $\Phi$.
\end{mlem}
\begin{proof}
	The proof is similar to that of Proposition \ref{prps_regV1}. Let us fix $\Phi:Z\to Z$ Lipschitz continuous with Lipschitz constant $Lip(\Phi)>0$. For $\mu_1,\mu_2\in\Delta(Z)$, let $\pi\in\Pi(\mu_1,\mu_2)$ be an optimal transport plan for $W_2(\mu_1,\mu_2)$. For any $(\alpha,\beta)\in\mathcal{A}_d\times\mathcal{B}_d$, one has
	\begin{align*}
		&\Big|\int_{Z}\bar J(\Phi(z),\alpha,\beta)d\mu_1(z)-\int_{Z}\bar J(\Phi(z'),\alpha,\beta)d\mu_2(z')\Big|\\
		\le &\int_{Z^2}\big|\bar J(\Phi(z),\alpha,\beta)-\bar J(\Phi(z'),\alpha,\beta)\big|d\pi(z,z').
	\end{align*}
	It follows further from Corollary \ref{cor_reg1} and the Cauchy-Schwarz inequality that,
	\begin{align*}
		\Big|\int_{{\R}^{n+1}}\bar J(z,\alpha,\beta)d\mu_1(z)-\int_{{\R}^{n+1}}\bar J(z',\alpha,\beta)d\mu_2(z')\Big|
		\le& C Lip(\Phi)\int_{Z^2}\|z-z'\|d\pi(z,z') \\
		\le& C Lip(\Phi) W_2(\mu_1,\mu_2)
	\end{align*}
	Since $(\alpha,\beta)\in\mathcal A_d\times\mathcal B_d$ in the above inequality is arbitrary, we apply Lemma \ref{lem_regtech} and get
	\begin{equation*}
		|\mathcal V^\pm(\Phi,\mu_1)-\mathcal V^\pm(\Phi,\mu_2)|\le Lip(\Phi)C W_2(\mu_1,\mu_2),
	\end{equation*}
	The proof is complete.
\end{proof}
Finally, we state the follow lemma which, together with Lemma \ref{cor_reg2}, implies the continuity of $\mathcal{V}^\pm: C(Z,Z)\times\Delta(Z)\to\R$. We refer interested readers to \cite{J&Q18} for a detailed proof of the result.
\begin{mlem}\label{lem_reg4}
	Under Assumptions \ref{Assum2}, for all $\Phi\in C(Z,Z)$, both the maps $\mu\mapsto \mathcal{V}^\pm(\Phi,\mu)$ are uniformly continuous on $\Delta(Z)$ with respect to the $W_2$-distance.
\end{mlem}

\subsection{Dynamic Programming Principle}
In this subsection, we aim to state and prove the dynamic programming principles for the extended value functions $\mathcal V^\pm$. To this end, we first analyse the status space of interest 
	$$\mathcal O(\mu_0):=\{\Phi\in C(Z;Z):\Phi(\supp{\mu_0})\subset \mathbb{R}^n\times(-\infty,M_0)\}.$$ 
Then we prove the relative compactness of the family of composed maps 
	$$\Phi_{Z,h}:=\{Z^{\cdot,u,v}_h\circ \Phi:u\in\mathcal U,\ v\in\mathcal V\}$$ 
	given $\Phi\in \mathcal O(\mu_0)$ and $h>0$, and we show that $\Phi_{Z,h}\subset\mathcal{O}(\mu_0)$ for $h>0$ sufficiently small before stating the dynamic programming principles. 
\begin{mrmk}
	If $\Phi\in\mathcal{O}(\mu_0)$, then $\mu_0$-almost surely no signal revelation occurs immediately after the game begins when $\Phi(z)$ is chosen as the initial position of the dynamics \eqref{dym_sys2}. As Corollary \ref{cor_DPP1} below implies, $\Phi\in\mathcal{O}(\mu_0)$ guarantees that the information structure of game $\mathcal{G}(\Phi\sharp\mu_0)$ will remain stable over short time intervals, which in turn allows us to develop the dynamic programming principle for the extended values $\mathcal{V}^\pm(\Phi,\mu_0)$.
\end{mrmk}
	Let $\mathcal{U}(t)$ (resp. $\mathcal{V}(t)$) denote the set of Lebesgue measurable controls $u:[0,t]\to\U$ (resp. $v:[0,t]\to \V$), and let $\mathcal{U}(t)$ and $\mathcal{V}(t)$ be equipped respectively with the corresponding $L^1$ norm. As direct a consequence of Assumptions \ref{Assum2}, we obtain by standard estimation methods and Gr\"{o}nwall's inequality that for all $t>0$ and $z_0\in Z$, the map
	\begin{equation*}
		Z^{z_0,\cdot}_t:\mathcal{U}(t)\times \mathcal{V}(t)\to Z: (u,v)\mapsto Z^{z_0,u,v}_t
	\end{equation*}
	is Lipschitz continuous. Furthermore, for any $\Phi\in C(Z,Z)$, the map $(u,v)\to Z^{\cdot,u,v}_t\circ \Phi$ is also Lipschitz continuous from $\mathcal{U}(t)\times \mathcal{V}(t)$ to $C(Z,Z)$. 
\begin{mlem}\label{lem_DPP1}
	The set $\mathcal O(\mu_0)$ is an open subset of $C(Z,Z)$.
\end{mlem}
\begin{proof}
	Let us fix $\Phi_0\in\mathcal O(\mu_0)$. We denote by $\Phi_0^{(n+1)}(z)$ the $(n+1)$-th coordinate of $\Phi_0(z)$ for $z\in Z$. Since $\text{supp}\ \mu_0$ is compact and the map $\Phi^{(n+1)}_0\in C(Z)$, 	
	$$M=\sup_{z\in \text{supp}\ \mu_0}\Phi^{(n+1)}_0(z)<M_0.$$

	Choose $\delta>0$ sufficiently small such that $M+\delta<M_0$. It follows that for any $\Phi\in B(\Phi_0;\delta)$ and for all $z\in\text{supp}\ \mu_0$,
	\begin{align*}
		\Phi^{(n+1)}(z)\le \Phi^{(n+1)}_0(z)+|\Phi^{(n+1)}(z)-\Phi^{(n+1)}_0(z)|
		\le M+\|\Phi(z)-\Phi_0(z)\|\le M+\delta<M_0.
	\end{align*}

Thus $\forall z\in \supp{\mu_0},\ \Phi(z)\in \mathbb R^n\times(-\infty,M_0)$, and $\Phi\in\mathcal O(\mu_0)$. Therefore $B(\Phi_0,\delta)\subset\mathcal O(\mu_0)$ and $\mathcal O(\mu_0)$ is an open subset of $C(Z,Z)$. The proof is complete.
\end{proof}
\begin{mcor}\label{cor_DPP1}
	Given $\Phi\in\mathcal O(\mu_0)$ and $h>0$ sufficiently small, $\overline{\Phi_{Z,h}}\subset \mathcal O(\mu_0)$.
\end{mcor}
\begin{proof}
	Let us denote by 
	$$\mathcal O_\delta(\mu_0):=\{\Psi\in C(Z,Z):\ \Psi(\supp{\mu_0})\subset\mathbb R^n\times(-\infty,M_0-\delta)\}\subset\mathcal O(\mu_0).$$ 		Since $\Phi\in\mathcal O(\mu_0)$ and $\Phi(\supp{\mu_0})$ is compact, there exists $\delta>0$ such that $\Phi\in\mathcal O_\delta(\mu_0)$. For any $(u,v)\in\mathcal U\times\mathcal V$ and for any $z\in \supp{\mu_0}$, 
	$$\Big|\Big(Z^{\cdot,u,v}_h\circ \Phi\Big)^{(n+1)}(z)-\Phi^{(n+1)}(z)\Big|\le\Big\|Z^{\cdot,u,v}_h\circ \Phi(z)- \Phi(z)\Big\|\le h\|F\|_\infty.$$
	Therefore, it suffices to choose 
	$$0<h<\frac{\delta}{2(\|F\|_\infty)},$$ 
	and one has $Z^{\cdot,u,v}_h\circ \Phi\in\mathcal O_{\delta/2}(\mu_0)$ and we deduce that $\Phi_{Z,h}\subset \mathcal O_{\delta/2}(\mu_0)\subset \overline{\mathcal O_{\delta/2}(\mu_0)}\subset\mathcal O(\mu_0)$ for $h>0$ small enough. The proof is complete.
\end{proof}
\begin{mlem}\label{lem_DPP2}
	For all $\Phi\in C(Z,Z)$ and $h>0$, the family $\Phi_{Z,h}$ is relatively compact in $C(Z,Z)$. In particular, for $\Phi\in\mathcal O(\mu_0)$ and $h>0 $ small enough, $\overline{\Phi_{Z,h}}\subset\mathcal{O}(\mu_0)$ is compact.
\end{mlem}
\begin{proof}
	In view of Corollary \ref{cor_DPP1}, it suffices to prove the first claim. Recall that by the Arzel\`{a}-Ascoli theorem, we only need to check that the family $\Phi_{Z,h}$ is uniformly equi-continuous and point-wise relatively compact. 
	
	Fix $\Phi\in\mathcal{O}(\mu_0)$ and $h>0$. For all $z,z'\in Z$ and $(u,v)\in\mathcal U\times\mathcal V$, one has
	$$\big\|Z^{\cdot,u,v}_h\circ \Phi(z)-Z^{\cdot,u,v}_h\circ \Phi(z')\big\|\le e^{L_Fh}\big\|\Phi(z)-\Phi(z')\big\|.$$
	Since $\Phi$ is uniformly continuous, for any $\varepsilon>0$, there exists $\delta>0$ such that 
	$$\|z-z'\|\le \delta\Longrightarrow\big\|\Phi(z)-\Phi(z')\big\|\le \varepsilon e^{-L_Fh}.$$
	Therefore, for any $z,z'\in Z$ and $(u,v)\in\mathcal U\times\mathcal V$,
	$$\|z-z'\|\le \delta\Longrightarrow \big\|Z^{\cdot,u,v}_h\circ \Phi(z)-Z^{\cdot,u,v}_h\circ \Phi(z')\big\|\le e^{L_Fh}\big\|\Phi(z)-\Phi(z')\big\|\le \varepsilon.$$
	Consequently, the family $\Phi_{Z,h}$ is uniformly equi-continuous. On the other hand, we note that for any $z\in Z$, the set
	$$\{Z^{\cdot,u,v}_h\circ \Phi(z):\ (u,v)\in\mathcal U\times\mathcal V\}\subset B(\Phi(z);h(\|F\|_\infty))$$
is bounded and thus relatively compact. Hence the family $\Phi_{Z,h}$ is point-wise relatively compact. The proof is complete.
\end{proof}
Now we are ready to state and prove the dynamic programming principles for the extended value functions.
\begin{mprps}\label{prps_DPP}
	Under Assumptions \ref{Assum2}, for any $\mu_0\in\Delta(Z)$, for all $\Phi\in\mathcal O(\mu_0)$ and $h>0$ sufficiently small such that $\Phi_{Z,h}\subset\mathcal O(\mu_0)$, one has
	\begin{equation}\label{eq_dpp}
		\begin{aligned}
			\mathcal V^+	(\Phi,\mu_0)=\inf_{\alpha\in\mathcal A_d}\sup_{v\in\mathcal V}\Big\{\int_{Z}d\mu_0(z)\Big[&\int^{h}_0e^{-\lambda t} \ell(Z^{\Phi(z),\alpha(v),v}_t,\alpha(v)(t),v(t))dt
\Big]\\
			&+e^{-\lambda h}\mathcal V^+(Z^{\cdot,\alpha(v),v}_h\circ\Phi,\mu_0)\Big]\Big\}.
		\end{aligned}
	\end{equation}
	and
	\begin{equation}
		\begin{aligned}
			\mathcal V^-	(\Phi,\mu_0)=\sup_{\beta\in\mathcal B_d}\inf_{u\in\mathcal U}\Big\{\int_{Z}d\mu_0(z)\Big[&\int^{h}_0e^{-\lambda t} \ell(Z^{\Phi(z),u,\beta(u)}_t,u(t),\beta(u)(t))dt
\Big]\\
			&+e^{-\lambda h}\mathcal V^+(Z^{\cdot,u,\beta(u)}_h\circ\Phi,\mu_0)\Big]\Big\}.
		\end{aligned}
	\end{equation}
\end{mprps}
\begin{proof}
	We only prove \eqref{eq_dpp} since the other equation can be established symmetrically. Let us fix $\mu_0\in\Delta(Z)$ and $\Phi\in\mathcal{O}(\mu_0)$. Consider $\mathcal{V}^+$ in its alternative form in Lemma \ref{lem_extvalue}:
	\begin{align*}
		\mathcal V^+	(\Phi,\mu_0)=\inf_{\alpha\in\mathcal A_d}	\sup_{v\in\mathcal V}\bar{\mathcal J}(\Phi,\mu_0,\alpha(v),v)
		=\inf_{\alpha\in\mathcal A_d}	\sup_{v\in\mathcal V}\int_{Z}\bar J(\Phi(z),\alpha(v),v)d\mu_0(z).
	\end{align*}
	Let us denote the right-hand side of \eqref{eq_dpp} by $\mathcal W^+(\Phi,\mu_0,h)$.
	\paragraph*{Step 1: $\mathcal V^+(\Phi,\mu_0)\ge\mathcal W^+(\Phi,\mu_0,h)$.} Let $\alpha^*\in\mathcal A_d$ be an $\varepsilon$-optimal strategy for $\mathcal V^+(\Phi,\mu_0)$.
	\begin{equation}\label{eq_dpp1}
		\begin{aligned}
			\mathcal V^+	(\Phi,\mu_0)\ge\sup_{v\in\mathcal V}\int_{Z}d\mu_0(z)\int^{\mathcal T(\Phi(z)_,\alpha^*(v),v)}_0e^{-\lambda t} \ell(Z^{\Phi(z),\alpha^*(v),v}_t,\alpha^*(v)(t),v(t))dt\\
			+e^{-\lambda {\mathcal T}(\Phi(z),\alpha^*(v),v)} V^+\Big(Z^{\Phi(z),\alpha^*(v),v}_{\mathcal T(\Phi(z),\alpha^*(v),v)}\Big)-\varepsilon.
		\end{aligned}
	\end{equation}
It follows that
	\begin{equation}\label{eq_dpp2}
		\begin{aligned}
			\mathcal W^+	(\Phi,\mu_0,h)\le\sup_{v\in\mathcal V}\int_{Z}d\mu_0(z)\Big[\int^{h}_0e^{-\lambda t} \ell(Z^{\Phi(z),\alpha^*(v),v}_t,\alpha^*(v)(t),v(t))dt\Big]
\\
			+e^{-\lambda h}\mathcal V^+\big(Z^{\cdot,\alpha^*(v),v}_h\circ\Phi,\mu_0\big).
		\end{aligned}
	\end{equation}
	Let us fix an arbitrary admissible control $v_0\in\mathcal V$ , and let us construct a new NAD strategy $\bar \alpha^*\in\mathcal A_d$ from $\alpha^*$ and $v_0$ by setting for all $v\in\mathcal V$, $\bar\alpha^*(v)=\alpha^*(\bar v)(\cdot+h)$ with
	$$\bar v(t)=\left\{\begin{aligned}
			&v_0(t),&t\in[0,h],\\
			&v(t-h), &t>h.
	\end{aligned}\right.$$
Therefore, by choosing an $\varepsilon$-optimal control $v^*\in\mathcal V$ for $\mathcal V^+\Big(Z^{\cdot,\alpha^*(v_0),v_0}_h\circ\Phi,\mu_0\Big)$ against $\bar\alpha^*$,
	\begin{equation}\label{eq_dpp3}
		\begin{aligned}
			\mathcal V^+	(\Phi,\mu_0)\ge
	\int_{Z}d\mu_0(z)\int^{\mathcal T(\Phi(z)_,\alpha^*(\bar v^*),\bar v^*)}_0e^{-\lambda t} \ell(Z^{\Phi(z),\alpha^*(\bar v^*),\bar v^*}_t,\alpha^*(\bar v^*)(t),\bar v^*(t))dt\\
			+e^{-\lambda {\mathcal T}(\Phi(z),\alpha^*(\bar v^*),\bar v^*)}V^+\Big(Z^{\Phi(z),\alpha^*(\bar v^*),\bar v^*}_{\mathcal T(\Phi(z),\alpha^*(\bar v^*),\bar v^*)}\Big)-\varepsilon
		\end{aligned}
	\end{equation}
	For any $z\in Z$, we have with $Z_h^\Phi:=Z^{\Phi(z),\alpha^*(v_0),v_0}_h$,
	\begin{equation}\label{eq_dpp4}
		\begin{aligned}
			&\int^{\mathcal T(\Phi(z)_,\alpha^*,\bar v^*)}_0e^{-\lambda t} \ell(Z^{\Phi(z),\alpha^*,\bar v^*}_t,\alpha^*(\bar v^*)(t),\bar v^*(t))dt\\
			=&\int^{h}_0e^{-\lambda t} \ell(Z^{\Phi(z),\alpha^*,\bar v^*}_t,\alpha^*(\bar v^*),\bar v^*)dt
			+\int^{\mathcal T(\Phi(z)_,\alpha^*,\bar v^*)}_he^{-\lambda t} \ell(Z^{\Phi(z),\alpha^*,\bar v^*}_t,\alpha^*(\bar v^*),\bar v^*)dt\\
			=&\int^{h}_0e^{-\lambda t} \ell(Z^{\Phi(z),\alpha^*,v_0}_t,\alpha^*(v_0),v_0)dt
			+e^{-\lambda h}\int^{\mathcal T(Z^\Phi_h,\bar \alpha^*,v^*)}_0 e^{-\lambda t} \ell(Z^{Z^\Phi_h,\bar \alpha^*,v^*}_t,\bar \alpha^*(v^*),v^*)dt.
		\end{aligned}
	\end{equation}
	Similarly,
	\begin{equation}\label{eq_dpp5}
		e^{-\lambda {\mathcal T}(\Phi(z),\alpha^*,\bar v^*)}V^+\Big(Z^{\Phi(z),\alpha^*,\bar v^*}_{\mathcal T(\Phi(z),\alpha^*,\bar v^*)}\Big)=e^{-\lambda h}\cdot e^{-\lambda {\mathcal T}\big(Z^\Phi_h,\bar \alpha^*,v^*\big)}V^+\Big(Z^{Z^\Phi_h,\bar \alpha^*,v^*}_{\mathcal T\big(Z^\Phi_h,\bar \alpha^*,v^*\big)}\Big)
	\end{equation}
	In combining \eqref{eq_dpp3}-\eqref{eq_dpp5}, we deduce obtain
	\begin{equation}\label{eq_dpp6}
		\begin{aligned}
			&\mathcal V^+(\Phi,\mu_0)\\
			\ge&\int_{Z}d\mu_0(z)\Big\{\int^{h}_0e^{-\lambda t} \ell(Z^{\Phi(z),\alpha^*,v_0}_t,\alpha^*(v_0),v_0)dt
			+e^{-\lambda h}\bar J\big(Z^{\Phi(z),\alpha^*,v_0}_h,\bar \alpha^*,v^*\big)\Big\}-\varepsilon\\
			\ge& \int_{Z}d\mu_0(z)\Big\{\int^{h}_0e^{-\lambda t} \ell(Z^{\Phi(z),\alpha^*,v_0}_t,\alpha^*(v_0),v_0)dt
			+e^{-\lambda h}\sup_{v\in\mathcal V}\bar J\big(Z^{\Phi(z),\alpha^*,v_0}_h,\bar\alpha^*,v\big)\Big\}-2\varepsilon\\
			\ge&\int_{Z}d\mu_0(z)\Big\{\int^{h}_0e^{-\lambda t} \ell(Z^{\Phi(z),\alpha^*,v_0}_t,\alpha^*(v_0),v_0)dt
			+e^{-\lambda h}\mathcal V^+(Z^{\cdot,\alpha^*,v_0}_h\circ\Phi,\mu_0)\Big\}-2\varepsilon.
		\end{aligned}
	\end{equation}
	Since $v_0\in\mathcal V$ in the last inequality above is arbitrary, one has
	\begin{equation}\label{eq_dpp7}
		\begin{aligned}
			&\mathcal V^+(\Phi,\mu_0)+2\varepsilon\\
			\ge &\sup_{v\in\mathcal V}\int_{Z}d\mu_0\int^{h}_0e^{-\lambda t} \ell(Z^{\Phi(z),\alpha^*(v),v}_t,\alpha^*(v),v)dt
	+e^{-\lambda h}\mathcal V^+(Z^{\cdot,\alpha^*(v),v}_h\circ\Phi,\mu_0)\\
			\ge &\mathcal W^+(\Phi,\mu_0,h).
		\end{aligned}
	\end{equation}
	Passing $\varepsilon\to0+$ on both sides of \eqref{eq_dpp7} and our first claim follows.
	\paragraph*{Step 2: $\mathcal V^+(\Phi,\mu_0)\le\mathcal W^+(\Phi,\mu_0,h)$.} We turn to the opposite inequality. Let $\tilde\alpha^*\in\mathcal A_d$ be an $\varepsilon$-optimal strategy for $\mathcal W^+(\Phi,\mu_0)$, namely,
	\begin{equation}\label{eq_dpp1B}
		\begin{aligned}
			\mathcal W^+	(\Phi,\mu_0)\ge\sup_{v\in\mathcal V}\int_{Z}d\mu_0(z)\Big[\int^{h}_0e^{-\lambda t} \ell(Z^{\Phi(z),\tilde\alpha^*(v),v}_t,\tilde\alpha^*(v),v)dt
\Big]\\
		+e^{-\lambda h}\mathcal V^+(Z^{\cdot,\tilde\alpha^*(v),v}_h\circ\Phi,\mu_0)-\varepsilon.
		\end{aligned}
	\end{equation}
	Let us fix an arbitrary $\varepsilon>0$, and let us choose, for all $\Psi\in\mathcal{O}(\mu_0)$, an $\varepsilon/2$-optimal strategy $\alpha_\Psi\in\mathcal A_d$ for $\mathcal V^+(\Psi,\mu_0)$. Hence, for all $\Psi\in\mathcal{O}(\mu_0)$,
	\begin{equation*}
		\mathcal V^+(\Psi,\mu_0)\ge \sup_{v\in\mathcal V}\bar{\mathcal{J}}(\Psi,\mu_0,\alpha_\Psi(v),v)-\frac{\varepsilon}{2}.
	\end{equation*}
	Since the map $\Psi\mapsto \mathcal V^+(\Psi,\mu_0)$ is continuous (Corollary \ref{cor_reg2}), for all $\Psi\in\mathcal{O}(\mu_0)$, there exists $\delta_\Psi>0$ such that $\alpha_\Psi$ is still $\varepsilon$-optimal for $\mathcal V^+(\Psi',\mu_0)$ for all $\Psi'\in B(\Psi;\delta_\Psi)\subset \mathcal O$. In addition, the collection $\{B(\Psi;\delta_\Psi/2)\}_{\Psi\in\overline{\Phi_{Z,h}}}$ forms an open cover of $\overline{\Phi_{Z,h}}$.
	
	By Lemma \ref{lem_DPP2}, the set $\Phi_{Z,h}:=\{Z^{\cdot,u,v}_h\circ\Phi\ |\ u\in\mathcal U,\ v\in\mathcal V\}$ is relatively compact. Therefore the above open cover has a finite open cover
	$$\bigcup_{k=1}^N B(\Psi_k;\delta_{\Psi_k}/2)\supset \overline{\Phi_{Z,h}}.$$
	For simplicity, we denote $B_k:=B(\Psi_k;\delta_{\Psi_k}/2)$ and $\alpha_k:=\alpha_{\Psi_k}$. Let us construct a Borel partition of $\mathcal O(\mu_0)$ by setting
	\begin{align*}
		E_0=\emptyset,\ E_{N+1}=\mathcal O(\mu_0)\backslash \bigcup_{j=1}^{N}B_j,\ 
		E_k=B_k\backslash \bigcup_{j=1}^{k-1}E_j,\ 1\le k\le N.
	\end{align*}
	Let $\tau>0$ be a common delay of $\tilde\alpha^*$ and $\{\alpha_k\}_{1\le k\le N}$. Without loss of generality, we can choose $\tau$ sufficiently small such that for all $(u,v)\in\mathcal U\times\mathcal V$,
	$$\|Z^{\cdot,u,v}_\tau-Id\|_\infty\le \delta:=\min_{1\le k\le N}\frac{\delta_{\Psi_k}}{2}.$$
	It follows that, for any $(u,v)\in\mathcal U\times\mathcal V$, if $\Psi\in E_k$ for some $1\le k\le N$, then $\alpha_k$ is an $\varepsilon$-optimal strategy for $\mathcal V^+(Z^{\cdot,u,v}_\tau\circ \Psi,\mu_0)$, i.e.
	\begin{equation}\label{eq_dpp8}
		\forall \Psi\in E_k,\ \forall (u,v)\in\mathcal U\times\mathcal V,\ \mathcal V^+(Z^{\cdot,u,v}_\tau\circ \Psi,\mu_0)\ge \sup_{v'\in\mathcal V}\bar{\mathcal J}(Z^{\cdot,u,v}_t\circ \Psi,\alpha_k(v'),v')-\varepsilon. 
	\end{equation}
	We proceed to construct a new strategy $\hat{\alpha}^*\in\mathcal A_d$ as follows. For all $v\in\mathcal V$ and $t\ge 0$, we write $v_{h,\tau}:=v|_{[h+\tau,\infty)}(\cdot+h+\tau)$ and define 
	\begin{equation}
		\hat\alpha^*(v)(t)=\left\{\begin{aligned}
			&\tilde\alpha^*(v)(t),&\txt{ if }t\in[0,h+\tau),\\
			&\tilde\alpha^*(v)(t),&\txt{ if }t\ge h+\tau\txt{ and }Z^{\cdot,\alpha^*(v),v}_h\circ\Phi\in E_{N+1},\\
			&\alpha_k(v_{h,\tau})(t-h-\tau),&\txt{ if }t\ge h+\tau\txt{ and }Z^{\cdot,\alpha^*(v),v}_h\circ\Phi\in E_k.&
		\end{aligned}\right.
	\end{equation}
	One can check that $\hat\alpha^*\in\mathcal A_d$, since the map $v\mapsto Z^{\cdot,\alpha^*(v),v}_h\circ\Phi$ from $\mathcal{V}$ to $C(Z,Z)$ is the composition of the measurable map $v\mapsto (\alpha^*(v),v)$ and the continuous map $(u,v)\mapsto Z^{\cdot,u,v}_h\circ\Phi$. We have furthermore,
	\begin{equation}\label{eq_dpp9}
		\begin{aligned}
			\mathcal V^+(\Phi,\mu)\le &\sup_{v\in \mathcal V}\bar{\mathcal J}(\Phi,\mu_0,\hat\alpha^*(v),v)\\
			=&\sup_{v\in \mathcal V}\int_{Z}d\mu_0(z)\Big\{\int^{\mathcal T(\Phi(z),\hat\alpha^*(v),v)}_0e^{-\lambda t} \ell(Z^{\Phi(z),\hat\alpha^*(v),v}_t,\hat\alpha^*(v)(t),v(t))dt\\
			&+e^{-\lambda {\mathcal T}(\Phi(z),\hat\alpha^*(v),v)}V^+\big(Z^{\Phi(z),\hat\alpha^*(v),v}_{\mathcal T(\Phi(z),\hat\alpha^*(v),v)}\big)\Big\}
		\end{aligned}
	\end{equation}
	For all $z\in Z$ and $v\in\mathcal V$, by definition of $\bar J$ and $\hat\alpha^*$, we observe that
	\begin{equation}\label{eq_dpp9A}
		\begin{aligned}
			&\int^{\mathcal T(\Phi(z),\hat\alpha^*,v)}_0e^{-\lambda t} \ell(Z^{\Phi(z),\hat\alpha^*,v}_t,\hat\alpha^*(v),v)dt\\
			=&\int^{h}_0e^{-\lambda t} \ell(Z^{\Phi(z),\tilde\alpha^*,v}_t,\tilde\alpha^*(v),v)dt+\int^{h+\tau}_h e^{-\lambda t} \ell(Z^{\Phi(z),\hat\alpha^*,v}_t,\hat\alpha^*(v),v)dt\\
			&+\int^{\mathcal T(\Phi(z),\hat\alpha^*,v)}_{h+\tau}e^{-\lambda t} \ell(Z^{\Phi(z),\hat\alpha^*,v}_t,\hat\alpha^*(v),v)dt\\
			\le &\int^{h}_0e^{-\lambda t} \ell(Z^{\Phi(z),\tilde\alpha^*,v}_t,\tilde\alpha^*(v),v)dt+\tau\|\ell\|_\infty
			+\int^{\mathcal T(\Phi(z),\hat\alpha^*,v)}_{h+\tau}e^{-\lambda t} \ell(Z^{\Phi(z),\hat\alpha^*,v}_t,\hat\alpha^*(v),v)dt,
		\end{aligned}
	\end{equation}
	and
	\begin{equation}\label{eq_dpp9B}
		\begin{aligned}
			&\int^{\mathcal T(\Phi(z),\hat\alpha^*,v)}_{h+\tau}e^{-\lambda t} \ell(Z^{\Phi(z),\hat\alpha^*,v}_t,\hat\alpha^*(v),v)dt+e^{-\lambda {\mathcal T}(\Phi(z),\hat\alpha^*,v)}V^+\big(Z^{\Phi(z),\hat\alpha^*,v}_{\mathcal T(\Phi(z),\hat\alpha^*,v)}\big)\\
			=&e^{-\lambda (h+\tau)}\sum^N_{k=1}\bar J(Z^{\cdot,\tilde\alpha^*(v),v}_{h+\tau}\circ\Phi(z),\alpha_k(v_{h,\tau}),v_{h,\tau})\textbf{1}_{E_k}(Z^{\cdot,\tilde\alpha^*(v),v}_{h}\circ \Phi).
		\end{aligned}
	\end{equation}
	Furthermore, we have
	\begin{equation}\label{eq_dpp9C}
		\begin{aligned}
			&\int_Zd\mu_0(z)\Big[e^{-\lambda (h+\tau)}\sum^N_{k=1}\bar J(Z^{\cdot,\tilde\alpha^*(v),v}_{h+\tau}\circ\Phi(z),\alpha_k(v_{h,\tau}),v_{h,\tau})\textbf{1}_{E_k}(Z^{\cdot,\tilde\alpha^*(v),v}_{h}\circ \Phi)\Big]\\
			= &e^{-\lambda (h+\tau)}\sum^N_{k=1}\textbf{1}_{E_k}(Z^{\cdot,\tilde\alpha^*(v),v}_{h}\circ \Phi)\int_Zd\mu_0(z)\Big[\bar J(Z^{\cdot,\tilde\alpha^*(v),v}_{h+\tau}\circ\Phi(z),\alpha_k(v_{h,\tau}),v_{h,\tau})\Big]\\
			\le &e^{-\lambda (h+\tau)}\sum^N_{k=1}\textbf{1}_{E_k}(Z^{\cdot,\tilde\alpha^*(v),v}_{h}\circ \Phi)\sup_{v'\in\mathcal V}\bar{\mathcal J}(Z^{\cdot,\tilde\alpha^*(v),v}_{h+\tau}\circ\Phi,\mu_0,\alpha_k(v'),v').
		\end{aligned}
	\end{equation}
	Substituting \eqref{eq_dpp9A}-\eqref{eq_dpp9C} into the right-hand side of \eqref{eq_dpp9} yields furthermore
	\begin{equation}\label{eq_dpp10}
		\begin{aligned}
			&\mathcal V^+(\Phi,\mu)\\
			\le &\sup_{v\in \mathcal V}\int_{Z}d\mu_0(z)\Big\{\int^{h}_0e^{-\lambda t} \ell(Z^{\Phi(z),\tilde\alpha^*(v),v}_t,\tilde\alpha^*(v)(t),v(t))dt\Big\}+\tau\|\ell\|_\infty\\
			&+e^{-\lambda (h+\tau)}\sum^N_{k=1}\textbf{1}_{E_k}(Z^{\cdot,\tilde\alpha^*(v),v}_{h}\circ \Phi)\int_{Z}d\mu_0\Big\{\bar J(Z^{\cdot,\tilde\alpha^*(v),v}_{h+\tau}\circ\Phi(z),\alpha_k(v_{h,\tau}),v_{h,\tau})\Big\}\\
			\le &\sup_{v\in \mathcal V}\int_{Z}d\mu_0(z)\Big\{\int^{h}_0e^{-\lambda t} \ell(Z^{\Phi(z),\tilde\alpha^*(v),v}_t,\tilde\alpha^*(v)(t),v(t))dt\Big\}+\tau\|\ell\|_\infty\\
			&+e^{-\lambda (h+\tau)}\sum^N_{k=1}\textbf{1}_{E_k}(Z^{\cdot,\tilde\alpha^*(v),v}_{h}\circ \Phi)\sup_{v'\in\mathcal V}\bar{\mathcal J}(Z^{\cdot,\tilde\alpha^*(v),v}_{h+\tau}\circ\Phi,\mu_0,\alpha_k(v'),v')
		\end{aligned}
	\end{equation}
	But $Z^{\cdot,\tilde\alpha^*(v),v}_{h+\tau}\circ\Phi=Z^{\cdot,\tilde\alpha^*(v)(\cdot+h),v(\cdot+h)}_{\tau}\circ\big(Z^{\cdot,\tilde\alpha^*(v),v}_{h}\circ\Phi\big)$, and it follows from \eqref{eq_dpp8} and Corollary \ref{cor_reg2} that, by choosing $\tau<\min\big\{\frac{\varepsilon}{\|\ell\|_\infty},\frac{\varepsilon}{C(1+\|F\|_\infty)}\big\}$ with $C>0$ given in Corollary \ref{cor_reg2}, we get
	\begin{equation}\label{eq_dpp11}
		\begin{aligned}
			&\sum^N_{k=1}\textbf{1}_{E_k}(Z^{\cdot,\tilde\alpha^*(v),v}_{h}\circ \Phi)\sup_{v'\in\mathcal V}\bar{\mathcal J}(Z^{\cdot,\tilde\alpha^*(v),v}_{h+\tau}\circ\Phi,\mu_0,\alpha_k(v'),v')\\
			\le& \sum^N_{k=1}\textbf{1}_{E_k}(Z^{\cdot,\tilde\alpha^*(v),v}_{h}\circ \Phi)\mathcal V^+(Z^{\cdot,\tilde\alpha^*(v),v}_{h+\tau}\circ\Phi,\mu_0)+\varepsilon\\
			\le& \mathcal V^+(Z^{\cdot,\tilde\alpha^*(v),v}_{h}\circ\Phi,\mu_0)+2\varepsilon
		\end{aligned}
	\end{equation}
	In combining \eqref{eq_dpp1B}, \eqref{eq_dpp10} and \eqref{eq_dpp11}, we obtain
	\begin{equation}\label{eq_dpp12}
		\begin{aligned}
			\mathcal V^+(\Phi,\mu)
			\le &\sup_{v\in \mathcal V}\int_{Z}d\mu_0(z)\Big\{\int^{h}_0e^{-\lambda t} \ell(Z^{\Phi(z),\tilde\alpha^*(v),v}_t,\tilde\alpha^*(v)(t),v(t))dt\Big\}+\tau\|\ell\|_\infty\\
			&+e^{-\lambda (h+\tau)}V^+(Z^{\cdot,\tilde\alpha^*(v),v}_{h}\circ\Phi,\mu_0)+2\varepsilon\\
			\le &\sup_{v\in \mathcal V}\int_{Z}d\mu_0(z)\Big\{\int^{h}_0e^{-\lambda t} \ell(Z^{\Phi(z),\tilde\alpha^*(v),v}_t,\tilde\alpha^*(v)(t),v(t))dt\Big\}\\
			&+e^{-\lambda h}\mathcal V^+(Z^{\cdot,\tilde\alpha^*(v),v}_{h}\circ\Phi,\mu_0)+3\varepsilon.\\
			\le &\mathcal W^+(\Phi,\mu_0)+4\varepsilon.
		\end{aligned}
	\end{equation}
	Since in the above inequality, $\varepsilon>0$ is arbitrary, letting $\varepsilon\to 0$ yields the desired result. The proof is complete.
\end{proof}

\section{Hamilton-Jacobi-Isaacs Equation and the Characterization of the Extended Value}
In this section, inspired by \cite{C&Q08,J&Q18}, we study the following Hamilton-Jacobi-Isaacs equation \eqref{HJB} and introduce the appropriate notion of its viscosity sub- and supersolutions on a given open subset of $C(Z,Z)$ (for notions and techniques regarding viscosity solutions in infinite-dimensional spaces, cf. \cite{CRANDALL1985379,CRANDALL1986368,CRANDALL1986214}). Our goal is to prove a comparison principle for \eqref{HJB} and, as a by-product, to obtain a characterization of $\mathcal{V}^\pm(\Phi,\mu_0)$ as the unique viscosity solution of the Hamilton-Jacobi-Isaacs equation on $\mathcal{O}(\mu_0)$ verifying given regularity conditions and a boundary condition.
\begin{equation}\label{HJB}
	-\lambda V(\Phi,\mu_0)+\mathcal H(\mu_0,\Phi,DV)=0,
\end{equation}
where the Hamiltonian $\mathcal H:\Delta(Z)\times C(Z,Z)\times C(Z,Z)\to\mathbb R$ is given by
\begin{align*}
	\mathcal H(\mu_0,\Phi,p_\Phi):=&\inf_{u_0\in\U}\sup_{v_0\in\V}\int_{Z}d\mu_0(z)\big[\ell\big(\Phi_0(z),u_0,v_0\big)+ F(\Phi_0(z),u_0,v_0),\cdot p_\Phi(z) \big]\\
	=&\sup_{v_0\in\V}\inf_{u_0\in\U}\int_{Z}d\mu_0(z)\big[\ell\big(\Phi_0(z),u_0,v_0\big)+ F(\Phi_0(z),u_0,v_0)\cdot p_\Phi(z)\big]
\end{align*}

\begin{mdef}
	Let $\delta>0$. $p_\Phi\in C(Z,Z)$ is said to belong to the $\delta$-superdifferential $D^+_\delta w(\Phi_0,\mu_0)$ to the function $w:C(Z,Z)\times\Delta(Z)\to\R$ at $(\Phi_0,\mu_0)\in C(Z,Z)\times \Delta(Z)$ iff
	$$\limsup_{\|\Psi\|_\infty\to 0}\frac{w(\Phi_0+\Psi,\mu_0)-w(\Phi_0,\mu_0)-\int_Z \Psi(z)\cdot p_\Phi(z) d\mu_0(z)}{\|\Psi\|_\infty}\le \delta,$$
	where the limit $\|\Psi\|_\infty\to 0$ should be understood in the sense of uniform convergence of $\Psi$ to $0$ in $C(Z,Z)$.
\end{mdef}

\begin{mdef}
	Let $\delta>0$. We say that $p_\Phi\in C(Z,Z)$ belongs to the $\delta$-subdifferential $D^-_\delta w(\Phi_0,\mu_0)$ to the function $w:C(Z,Z)\times\Delta(Z)\to\R$ at $(\Phi_0,\mu_0)\in C(Z,Z)\times \Delta(Z)$ iff
	$$\liminf_{\|\Psi\|_\infty\to 0}\frac{w(\Phi_0+\Psi,\mu_0)-w(\Phi_0,\mu_0)-\int_Z \Psi(z)\cdot p_\Phi(z) d\mu_0(z)}{\|\Psi\|_\infty}\ge -\delta.$$
\end{mdef}

\begin{mdef}
	For any $\mu_0\in\Delta(Z)$ and $O\subset C(Z,Z)$ an open subset, the function $w:C(Z,Z)\times\Delta(Z)\to\R$ is a viscosity subsolution to \eqref{HJB} on $O$ if and only if there exists $C>0$ such that, for all $\Phi_0\in O$ and for all $p_\Phi\in D^+_\delta w(\Phi_0,\mu_0)$, 
	$$-\lambda w+\mathcal{H}(\mu_0,\Phi_0,p_\Phi)\ge -C\delta.$$
\end{mdef}

\begin{mdef}
	For any $\mu_0\in\Delta(Z)$ and $O\subset C(Z,Z)$ an open subset, the function $w:C(Z,Z)\times\Delta(Z)\to\R$ is a viscosity supersolution to \eqref{HJB} on $O$ if and only if there exists $C>0$ such that, for all $\Phi_0\in O$ and for all $p_\Phi\in D^-_\delta w(\Phi_0,\mu_0)$, 
	$$-\lambda w+\mathcal{H}(\mu_0,\Phi_0,p_\Phi)\le C\delta.$$
\end{mdef}

\begin{mlem}\label{lem_subvis}
	For any $\mu_0\in\Delta(Z)$ and $O\subset C(Z,Z)$ an open subset, if $w$ is a viscosity subsolution to \eqref{HJB} on $O$, then $-w$ is a viscosity supersolution to the equation \eqref{HJB2}
	\begin{equation}\label{HJB2}
		-\lambda V(\Phi,\mu_0)+\tilde{\mathcal{H}}(\mu_0,\Phi,DV)=0,
	\end{equation}
	on $O$ with $\tilde{\mathcal{H}}(\mu_0,\Phi,p_\Phi):=-\mathcal{H}(\mu_0,\Phi,-p_\Phi)$.
\end{mlem}
\begin{proof}
	It suffices to observe that for any $\Phi_0\in O$, if $p_\Phi\in D^-_\delta (-w)(\Phi_0,\mu_0)$, then
	$$\limsup_{\|\Psi\|_\infty\to 0}\frac{w(\Phi_0+\Psi,\mu_0)-w(\Phi_0,\mu_0)-\int_Z \Psi(z)\cdot(-p_\Phi(z)) d\mu_0(z)}{\|\Psi\|_\infty}\le \delta,$$
	and consequently $-p_\Phi\in D^+_\delta w(\Phi_0,\mu_0)$. Moreover, since $w$ is a viscosity subsolution to \eqref{HJB} on $O$, it follows that there exists $C>0$ independent of the choice of $\Phi_0\in O$ and $-p_\Phi\in D^+_\delta w(\Phi_0,\mu_0)$ such that
	\begin{align*}
		-C\delta \le& -\lambda w+\mathcal{H}(\mu_0,\Phi_0,-p_\Phi)\\
		=& \lambda (-w)+\sup_{v_0\in\V}\inf_{u_0\in\U}\int_{Z}d\mu_0(z)\big[\ell\big(\Phi_0(z),u_0,v_0\big)+ F(\Phi_0(z),u_0,v_0)\cdot(-p_\Phi(z))\big]\\
		=& \lambda (-w)-\inf_{u_0\in\U}\sup_{v_0\in\V}\int_{Z}d\mu_0(z)\big[-\ell\big(\Phi_0(z),u_0,v_0\big)+ F(\Phi_0(z),u_0,v_0)\cdot p_\Phi(z)\big].
	\end{align*}		
	Therefore, we have
	$$-\lambda(-w)+\tilde{\mathcal{H}}(\mu_0,\Phi_0,p_\Phi)\le C\delta.$$
	The proof is complete.
\end{proof}
	We show that $\mathcal{V}^+$ is a viscosity subsolution to \eqref{HJB} on the open set $\mathcal{O}(\mu_0)$ in the proposition below by dividing both sides of \eqref{eq_dpp} in the dynamic programming principle by $$\big\|Z^{\Phi_0(\cdot),\alpha(v),v}_h-\Phi_0\big\|_{\infty}$$ and then letting $h>0$ tends to $0$. Moreover, we deduce from Lemma \ref{lem_subvis}, $\mathcal{V}^-$ is a viscosity supersolution to \eqref{HJB} on $\mathcal{O}(\mu_0)$ (see Corollary \ref{cor_supvis}).
\begin{mprps}\label{prps_subvis}
	For any $\mu_0\in\Delta(Z)$, the extended upper value function $\mathcal{V}^+(\mu,\Phi)$ is a viscosity subsolution to \eqref{HJB} on
	$\mathcal O(\mu_0)$.
\end{mprps}
\begin{proof}
	Fix $\Phi_0\in \mathcal{O}(\mu_0)$ and let $p_\Phi\in D^+_\delta \mathcal{V}^+(\Phi_0,\mu_0)$. We have, for all $t\ge 0$ and for all $(\alpha,v)\in\mathcal{A}_d\times\mathcal{V}$,
	\begin{equation}\label{eq_prps_subvis1}
		\begin{aligned}
			&\mathcal{V}^+(\Phi_0,\mu_0)-\mathcal{V}^+(Z^{\Phi_0(\cdot),\alpha(v),v}_t,\mu_0)+\int_X \big(Z^{\Phi_0(\cdot),\alpha(v),v}_t-\Phi_0\big)(z)\cdot p_\Phi(z) d\mu_0(z)\\
			\ge&\big\|Z^{\Phi_0(\cdot),\alpha(v),v}_t-\Phi_0\big\|_{\infty}\big[-\delta-\varepsilon\big(\big\|Z^{\Phi_0(\cdot),\alpha(v),v}_t-\Phi_0\big\|_\infty\big)\big]
		\end{aligned}
	\end{equation}
	where $\varepsilon(\gamma)\to 0$ as $\gamma\to 0+$. Since
	$$Z^{\Phi_0(z),\alpha(v),v}_t=\Phi_0(z)+\int^t_0 F\big(Z^{\Phi_0(z),\alpha(v),v}_s,\alpha(v)(s),v(s)\big)ds,$$
	we denote by $F(s,z,\Phi_0,\alpha,v):=F\big(Z^{\Phi_0(z),\alpha(v),v}_s,\alpha(v)(s),v(s)\big)$, and \eqref{eq_prps_subvis1} can be rewritten as
	\begin{equation}\label{eq_prps_subvis2}
		\begin{aligned}
			&\int_Z\int^t_0 F(s,z,\Phi_0,\alpha,v)\cdot p_\Phi(z) dsd\mu_0(z)+\mathcal{V}^+(\Phi_0,\mu_0)-\mathcal{V}^+(Z^{\Phi_0(\cdot),\alpha(v),v}_t,\mu_0)\\
			\ge&\big\|Z^{\Phi_0(\cdot),\alpha(v),v}_t-\Phi_0\big\|_{\infty}\big[-\delta-\varepsilon\big(\big\|Z^{\Phi_0(\cdot),\alpha(v),v}_t-\Phi_0\big\|_\infty\big)\big]
		\end{aligned}
	\end{equation}
	But $F$ is bounded, thus $\big\|Z^{\Phi_0(\cdot),\alpha(v),v}_t-\Phi_0\big\|_{\infty}\le t\|F\|_\infty$, and it follows from \eqref{eq_prps_subvis2} that
	\begin{equation}\label{eq_prps_subvis2B}
		\begin{aligned}
			&\mathcal{V}^+(Z^{\Phi_0(\cdot),\alpha(v),v}_t,\mu_0)\\
	\le &\int_Z\int^t_0 F(s,z,\Phi_0,\alpha,v)\cdot p_\Phi(z) dsd\mu_0(z)+\mathcal{V}^+(\Phi_0,\mu_0)+ t\|F\|_\infty\big[\delta+\varepsilon(t\|F\|_\infty)\big]
		\end{aligned}
	\end{equation}
	By the dynamic programming principle (Proposition \ref{prps_DPP}), for $t>0$ sufficiently small,
	\begin{equation}\label{eq_prps_subvis3}
		\begin{aligned}
			\mathcal{V}^+(\Phi_0,\mu_0)=&\inf_{\alpha\in\mathcal A_d}\sup_{v\in\mathcal V}\int_{Z}d\mu_0(z)\Big[\int^{t}_0e^{-\lambda s} \ell(Z^{\Phi_0(z),\alpha(v),v}_s,\alpha(v)(s),v(s))ds
\Big]\\
			&+e^{-\lambda t}\mathcal{V}^+(Z^{\Phi_0(\cdot),\alpha(v),v}_t,\mu_0).
		\end{aligned}
	\end{equation}
	Substituting \eqref{eq_prps_subvis2B} into the last term on the right-hand side of the above inequality yields 
	\begin{equation}\label{eq_prps_subvis4}
		\begin{aligned}
			(1-e^{-\lambda t})\mathcal{V}^+(\Phi_0,\mu_0)		\le&\inf_{\alpha\in\mathcal{A}_d}\sup_{v\in\mathcal{V}}\int_{Z}d\mu_0(z)\Big\{\int^{t}_0e^{-\lambda s} \ell(Z^{\Phi_0(z),\alpha(v),v}_s,\alpha(v)(s),v(s))ds
\\
			&+e^{-\lambda t}\int^t_0 F(s,z,\Phi_0,\alpha,v)\cdot p_\Phi(z)ds\Big\}+ t\|F\|_\infty\big[\delta+\varepsilon(t\|F\|_\infty)\big]
		\end{aligned}
	\end{equation}
	In particular, let us fix $u_0\in\U\subset\mathcal{A}_d$ a constant control and let $v\in\mathcal V$ be $\varepsilon t$-optimal against $u_0\in\mathcal A_d$ on the right-hand side of \eqref{eq_prps_subvis4}. We obtain
	\begin{equation}\label{eq_prps_subvis5}
		\begin{aligned}
			(1-e^{-\lambda t})\mathcal{V}^+(\Phi_0,\mu_0)\le&\int_{Z}d\mu_0(z)\Big\{\int^{t}_0e^{-\lambda s} \ell(Z^{\Phi_0(z),u_0,v}_s,u_0,v(s))ds
\\
			&+e^{-\lambda t}\int^t_0 F(s,z,\Phi_0,u_0,v)\cdot p_\Phi(z) ds\Big\}+ t\|F\|_\infty\big[\delta+\varepsilon(t\|F\|_\infty)\big]
		\end{aligned}
	\end{equation}
	
	Now we estimate the terms on both sides of the above inequality. Clearly, one has
	\begin{equation}\label{eq_prps_subvis6}
		(1-e^{-\lambda t})\mathcal{V}^+(\Phi_0,\mu_0)=\lambda t\mathcal{V}^+(\Phi_0,\mu_0)+o(t).
	\end{equation}
	By the regularity of the trajectory $t\mapsto Z_t$ and the Lipschitz continuity of $F$ and $\ell$, for the constant $L=\max(L_F,L_\ell)$ independent of the choice of $t>0$, we have $\forall s\in [0,t]$ and $\forall z\in Z$, 
	\begin{equation}\label{eq_prps_subvis7}
		\begin{aligned}
			\big|\ell\big(Z^{\Phi_0(z),u_0,v}_s,u_0,v(s))-\ell(\Phi_0(z),u_0,v(s)\big)\big|&\le Lt\|F\|_\infty;\\
			\big\|F(Z^{\Phi_0(z),u_0,v}_s,u_0,v(s))-F(\Phi_0(z),u_0,v(s))\big\|&\le Lt\|F\|_\infty.
		\end{aligned}
	\end{equation}
	Therefore, $\forall z\in Z$,
	\begin{equation}
		\begin{aligned}
			\int^{t}_0e^{-\lambda s} \ell\big(Z^{\Phi_0(z),u_0,v}_s,u_0,v(s)\big)ds&\le \int^{t}_0\ell\big(\Phi_0(z),u_0,v(s)\big)ds+o(t);\\
			e^{-\lambda t}\int^t_0 F(s,z,\Phi_0,u_0,v)\cdot p_\Phi(z) ds&\le \int^t_0 F(\Phi_0(z),u_0,v(s))\cdot p_\Phi(z) ds+o(t).
		\end{aligned}
	\end{equation}
	The above two inequalities together yield furthermore
	\begin{equation}\label{eq_prps_subvis8}
		\begin{aligned}
			&\int_{Z}d\mu_0(z)\Big\{\int^{t}_0e^{-\lambda s} \ell\big(Z^{\Phi_0(z),u_0,v}_s,u_0,v(s)\big)ds+e^{-\lambda t}\int^t_0  F(s,\Phi_0,u_0,v)\cdot p_\Phi(z)ds\Big\}\\
			\le& \int_{Z}d\mu_0(z)\Big\{\int^{t}_0\ell\big(\Phi_0(z),u_0,v(s)\big)+ F(\Phi_0(z),u_0,v(s))\cdot p_\Phi(z) ds\Big\}+o(t)\\
			\le& t\sup_{v_0\in\mathbb V}\int_{Z}d\mu_0(z)\big[\ell\big(\Phi_0(z),u_0,v_0\big)+ F(\Phi_0(z),u_0,v_0)\cdot p_\Phi(z)\big]+o(t)
		\end{aligned}
	\end{equation}
	In combining \eqref{eq_prps_subvis5}-\eqref{eq_prps_subvis8}, we obtain, 
	\begin{equation}\label{eq_prps_subvis9}
		\begin{aligned}
			\lambda t\mathcal{V}^+(\Phi_0,\mu_0)\le&t\sup_{v_0\in\mathbb V}\int_{X}d\mu_0(z)\big[\ell\big(\Phi_0(z),u_0,v_0\big)+ F(\Phi_0(z),u_0,v_0)\cdot p_\Phi(z)\big]\\
		&+o(t)+ t\|F\|_\infty\big[\delta+\varepsilon(t\|F\|_\infty)\big]
\end{aligned}
	\end{equation}
	But $u_0\in\mathbb U$ in the above inequality is arbitrary, and taking the infimum of both sides of \eqref{eq_prps_subvis9} over $u_0\in\U$ yields
	\begin{equation}
		\begin{aligned}
			\lambda t\mathcal{V}^+(\Phi_0,\mu_0)\le&t\inf_{u_0\in\mathbb U}\sup_{v_0\in\mathbb V}\int_{Z}d\mu_0(z)\big[\ell\big(\Phi_0(z),u_0,v_0\big)+ F\big(\Phi_0(z), u_0,v_0\big)\cdot p_\Phi(z)\big]\\
			&+o(t)+ t\|F\|_\infty\big[\delta+\varepsilon(t\|F\|_\infty)\big]
		\end{aligned}
	\end{equation}
	Finally, we divide both sides of the above inequality by $t$ and then pass $t\to 0+$ to obtain the desired inequality:
	\begin{equation}
		-\lambda \mathcal{V}^+(\Phi_0,\mu_0)+\inf_{u_0\in\mathbb U}\sup_{v_0\in\V}\int_{Z}d\mu_0(z)\big[\ell\big(\Phi_0(z),u_0,v_0\big)+ F(\Phi_0(z),u_0,v_0)\cdot p_\Phi(z)\big]\ge -C\delta,
	\end{equation}
	where $C=\|F\|_\infty$. The proof is complete.
\end{proof}

\begin{mcor}\label{cor_supvis}
	For any $\mu_0\in\mathcal{P}(Z)$, the extended lower value function $\mathcal{V}^-(\mu,\Phi)$ is a viscosity supersolution to \eqref{HJB} on
	$\mathcal O:=\{\Phi\in C(Z;\mathbb{R}^n):\Phi(\supp{\mu_0})\subset \mathbb{R}^n\times(-\infty,M_0)\}.$
\end{mcor}
\begin{proof}
	It suffices to see that $-\mathcal{V}^-(\Phi_0,\mu_0)$ is the extended upper value of another differential game with the same dynamics and game procedure as $\mathcal{G}(\mu_0)$ and with the running cost
	$$\int^\infty_0e^{-\lambda t}\big[-\ell\big(z(t),u(t),v(t)\big)\big]dt.$$
	Therefore by Proposition \ref{prps_subvis}, for $\mu_0\in\mathcal{P}(Z)$, $-\mathcal{V}^-$ is a viscosity subsolution to \eqref{HJB2} on $\mathcal{O}$. Applying Lemma \ref{lem_subvis}, we obtain that $\mathcal{V}^-$ is a viscosity super solution to \eqref{HJB} on $\mathcal{O}$. The proof is complete.
\end{proof}
Next, we establish the following comparison principle for the Hamilton-Jacobi-Isaacs equation \eqref{HJB}.
\begin{mprps}\label{prps_CP}
	For $\mu_0\in\Delta(Z)$, let $W_1:C(Z,Z)\times\Delta(Z)\to \R$ and $W_2:C(Z,Z)\times\Delta(Z)\to \R$ be respectively a viscosity subsolution and a viscosity supersolution to \eqref{HJB} on $O$. We assume that:
	\begin{enumerate}[label=(H\arabic*)]
		\item both $W_1$ and $W_2$ are bounded continuous;\label{CP_H1}
		\item for any $\mu\in\Delta(Z)$, both maps is $k$-Lipschitz continuous with $k>0$, namely for $i=1,2$,\label{CP_H2}
			$$|W_i(\Phi,\mu)-W_i(\Psi,\mu)|\le k\|\Phi-\Psi\|_{L^2_\mu};$$
		\item for any $\Phi\in O^c$, $W_1(\mu_0,\Phi)\le W_2(\mu_0,\Phi)$.\label{CP_H3}
	\end{enumerate}
	Then for any $\Phi\in O$, $W_1(\Phi,\mu_0)\le W_2(\Phi,\mu_0)$.
\end{mprps}
\begin{proof}
	We prove the proposition by contradiction. Assume that there exists some $\alpha>0$ and some $\Phi_0\in O$ such that
	\begin{equation}\label{eq_cpA}
		W_2(\Phi_0,\mu_0)-W_1(\Phi_0,\mu_0)\le-\alpha.
	\end{equation}
	By the regularity condition \ref{CP_H2} and the boundary condition \ref{CP_H3}, we deduce that the set 
	$$\mathcal F_{\alpha}(\mu_0):=\{\Phi\in C(Z,Z):\ W_2(\Phi,\mu_0)-W_1(\Phi,\mu_0)\le -\alpha\}\subset O$$ 
	is closed, and in addition, for some $\delta>0$,
	$$\inf \{\|\Phi-\Psi\|_{L^2_{\mu_0}}: \Phi\in \mathcal F_{\alpha}(\mu_0),\ \Psi\in O^c\}>\delta.$$ 
	We employ the double-variable technique by setting, for $\varepsilon>0$ sufficiently small,
	\begin{equation*}
		\begin{aligned}
			W_\varepsilon:C(Z,Z)^2&\to\mathbb R\\
			(\Phi_1,\Phi_2)&\mapsto W_\varepsilon(\Phi_1,\Phi_2)=W_2(\Phi_2,\mu_0)-W_1(\Phi_1,\mu_0)+\frac{1}{\varepsilon}\|\Phi_1-\Phi_2\|^2_{L^2_{\mu_0}}
		\end{aligned}
	\end{equation*}
	Notice that with $W_1$, $W_2$ both bounded, the map $W_\varepsilon$ is bounded from below and 
	$$-\infty<\inf_{(\Phi_1,\Phi_2)\in C(Z,Z)^2} W_\varepsilon(\Phi_1,\Phi_2)\le W_\varepsilon(\Phi_0,\Phi_0) \le -\alpha.$$
	Since $C(Z,Z)^2$ is a complete metric space while $C(Z,Z)$ is equipped with the infinity norm, and $W_\varepsilon$ is continuous, by Ekeland's variational principle (cf. \cite{EKELAND74}), there exists a pair $(\bar\Phi_1,\bar\Phi_2)\in C(Z,Z)^2$ such that
	\begin{enumerate}[label=(E\arabic*)]
		\item $W_\varepsilon(\bar\Phi_1,\bar\Phi_2)\le W_\varepsilon(\Phi_0,\Phi_0)$;\label{CP_E1}
		\item for all $(\Phi_1,\Phi_2)\in C(Z,Z)^2$, 		
		$W_\varepsilon(\bar\Phi_1,\bar\Phi_2)\le W_\varepsilon(\Phi_1,\Phi_2)+\varepsilon(\|\Phi_1-\bar\Phi_1\|_\infty+\|\Phi_2-\bar\Phi_2\|_\infty).$\label{CP_E2}
	\end{enumerate}
	\paragraph*{Step 1.} Let us show that, for $\varepsilon>0$ sufficiently small, $\bar\Phi_1\in O$ and $\bar\Phi_2\in O$.\\
	Applying \ref{CP_E2} with $(\Phi_1,\Phi_2)=(\bar\Phi_2,\bar\Phi_2)$ yields
	\begin{equation}
		-W_1(\bar\Phi_1,\mu_0)+\frac{1}{\varepsilon}\|\bar\Phi_1-\bar\Phi_2\|^2_{L^2_{\mu_0}}\le -W_1(\bar\Phi_2,\mu_0)+\varepsilon\|\bar\Phi_2-\bar\Phi_1\|_\infty.
	\end{equation}
	Without loss of generality, we choose the Lipschitz constant $k\ge 2\max_{z\in Z}\|z\|$ in \ref{CP_H2}, and the above inequality further implies
	\begin{equation}
		\frac{1}{\varepsilon}\|\bar\Phi_1-\bar\Phi_2\|^2_{L^2_{\mu_0}}\le W_1(\bar\Phi_1,\mu_0)-W_1(\bar\Phi_2,\mu_0)+k\varepsilon\le k \|\bar\Phi_1-\bar\Phi_2\|_{L^2_{\mu_0}}+k\varepsilon,
	\end{equation}
	Consequently, $q=\|\bar\Phi_1-\bar\Phi_2\|_{L^2_{\mu_0}}$ verifies
	$$\frac{1}{\varepsilon}q^2-kq-k\varepsilon\le 0,$$
	which in turn yields:
	\begin{equation}\label{eq_cp0}
		\|\bar\Phi_1-\bar\Phi_2\|_{L^2_{\mu_0}}\le \frac{\varepsilon}{2} (k+\sqrt{k^2+4k}).
	\end{equation}		
	Therefore, it suffices to choose $\varepsilon\le (k^2+k\sqrt{k^2+4k})^{-1}\alpha $, and we have
	\begin{equation}\label{eq_cp1}
		\begin{aligned}
			-\alpha\ge W_\varepsilon(\bar\Phi_1,\bar\Phi_2)\ge& W_2(\bar\Phi_2,\mu_0)-W_1(\bar\Phi_1,\mu_0)\\
			\ge& W_2(\bar\Phi_1,\mu_0)-W_1(\bar\Phi_1,\mu_0)-\frac{\varepsilon}{2} k(k+\sqrt{k^2+4k})\\
			\ge& W_2(\bar\Phi_1,\mu_0)-W_1(\bar\Phi_1,\mu_0)-\frac{\alpha}{2}.
		\end{aligned}
	\end{equation}		
	Hence $\bar\Phi_1\in \mathcal F_{\alpha/2}(\mu_0)\subset O$. Exchanging the roles of $\bar\Phi_1$ and $\bar\Phi_2$ in \eqref{eq_cp1}, we obtain $\bar\Phi_2\in O$.
	\paragraph*{Step 2.}
	Let us show that $\frac{2}{\varepsilon}(\bar\Phi_1-\bar\Phi_2)\in D^+_\varepsilon W_1(\bar\Phi_1,\mu_0)$. By \ref{CP_E2}, for any $\Phi\in C(Z,Z)$,
	\begin{equation}
		W_\varepsilon(\bar\Phi_1,\bar\Phi_2)\le W_\varepsilon(\Phi,\bar\Phi_2)+\varepsilon\|\Phi-\bar\Phi_1\|_\infty,
	\end{equation}
	and consequently,
	\begin{equation}
		-W_1(\bar\Phi_1,\mu_0)+\frac{1}{\varepsilon}\|\bar\Phi_1-\bar\Phi_2\|^2_{L^2_{\mu_0}}
		\le -W_1(\Phi,\mu_0)+\frac{1}{\varepsilon}\|\Phi-\bar\Phi_2\|^2_{L^2_{\mu_0}}+\varepsilon\|\Phi-\bar\Phi_1\|_\infty.
	\end{equation}	
	From the above inequality, we deduce that
	\begin{equation}
		\begin{aligned}
			&\varepsilon\|\Phi-\bar\Phi_1\|_\infty\ge W_1(\Phi,\mu_0)-W_1(\bar\Phi_1,\mu_0)+\frac{1}{\varepsilon}\big[\|\bar\Phi_1-\bar\Phi_2\|^2_{L^2_{\mu_0}}-\|\Phi-\bar\Phi_2\|^2_{L^2_{\mu_0}}\big]\\
			=& W_1(\Phi,\mu_0)-W_1(\bar\Phi_1,\mu_0)+\frac{1}{\varepsilon}\big[-\|\Phi-\bar\Phi_1\|^2_{L^2_{\mu_0}}-2\langle \Phi-\bar\Phi_1,\bar\Phi_1-\bar\Phi_2\rangle_{L^2_{\mu_0}}\big]
		\end{aligned}
	\end{equation}	
	Dividing both sides of the last inequality above by $\|\Phi-\bar\Phi_1\|_\infty$ yields
	\begin{equation}\label{eq_cp2}
		\begin{aligned}
			&\frac{W_1(\Phi,\mu_0)-W_1(\bar\Phi_1,\mu_0)-\frac{2}{\varepsilon}\langle \Phi-\bar\Phi_1,\bar\Phi_1-\bar\Phi_2\rangle_{L^2_{\mu_0}}}{\|\Phi-\bar\Phi_1\|_\infty}\\
			&\le \varepsilon+\frac{\|\Phi-\bar\Phi_1\|^2_{L^2_{\mu_0}}}{\varepsilon\|\Phi-\bar\Phi_1\|_\infty}
			\le \varepsilon+\frac{\|\Phi-\bar\Phi_1\|^2_{\infty}}{\varepsilon\|\Phi-\bar\Phi_1\|_\infty}
			= \varepsilon+\frac{\|\Phi-\bar\Phi_1\|_{\infty}}{\varepsilon}
		\end{aligned}
	\end{equation}	
	Passing $\|\Phi-\bar\Phi_1\|_{\infty}\to 0+$ on both sides of \eqref{eq_cp2}, we obtain:
	\begin{equation}\label{eq_cp3}
		\limsup_{\|\Phi-\bar\Phi_1\|_{\infty}\to 0}\frac{W_1(\Phi,\mu_0)-W_1(\bar\Phi_1,\mu_0)-\frac{2}{\varepsilon}\langle \Phi-\bar\Phi_1,\bar\Phi_1-\bar\Phi_2\rangle_{L^2_{\mu_0}}}{\|\Phi-\bar\Phi_1\|_\infty}\\
		\le \varepsilon.
	\end{equation}
	Therefore, $\frac{2}{\varepsilon}(\bar\Phi_1-\bar\Phi_2)\in D^+_\varepsilon W_1(\bar\Phi_1,\mu_0)$. As $W_1$ is a viscosity subsolution to \eqref{HJB} on $O$, we denote by $C>0$ the constant such that both inequalities for the sub- and supersolutions hold, and we have
	\begin{equation}\label{eq_cp4}
		\begin{aligned}
			-\lambda W_1(\bar\Phi_1,\mu_0)+\inf_{u_0\in\mathbb U}\sup_{v_0\in\V}\int_{Z}d\mu_0(z)\big[\ell\big(\bar\Phi_1(z),u_0,v_0\big)\\
			+\frac{2}{\varepsilon} F(\bar\Phi_1(z),u_0,v_0)\cdot\big(\bar\Phi_1(z)-\bar\Phi_2(z)\big)\big]\ge -C\varepsilon.
		\end{aligned}
	\end{equation}
	
	Similarly, for any $\Psi\in C(Z,Z)$, applying \ref{CP_E2} with $(\Phi_1,\Phi_2)=(\bar\Phi_1,\Psi)$ yields
	\begin{equation}
		W_2(\bar\Phi_2,\mu_0)+\frac{1}{\varepsilon}\|\bar\Phi_1-\bar\Phi_2\|^2_{L^2_{\mu_0}}
		\le W_2(\Psi,\mu_0)+\frac{1}{\varepsilon}\|\Psi-\bar\Phi_1\|^2_{L^2_{\mu_0}}+\varepsilon\|\Psi-\bar\Phi_2\|_\infty.
	\end{equation}	
	The above inequality implies in turn
	\begin{equation}
		\begin{aligned}
			&-\varepsilon\|\Psi-\bar\Phi_2\|_\infty\le W_2(\Psi,\mu_0)-W_2(\bar\Phi_2,\mu_0)+\frac{1}{\varepsilon}\big[\|\Psi-\bar\Phi_1\|^2_{L^2_{\mu_0}}-\|\bar\Phi_1-\bar\Phi_2\|^2_{L^2_{\mu_0}}\big]\\
			=& W_2(\Psi,\mu_0)-W_2(\bar\Phi_2,\mu_0)+\frac{1}{\varepsilon}\big[\|\Psi-\bar\Phi_2\|^2_{L^2_{\mu_0}}+2\langle \Psi-\bar\Phi_2,\bar\Phi_2-\bar\Phi_1\rangle_{L^2_{\mu_0}}\big]
		\end{aligned}
	\end{equation}	
	As in \eqref{eq_cp2} and \eqref{eq_cp3}, we divide both sides of last inequality above by $\|\Psi-\bar\Phi_2\|_{L^2_{\mu_0}}$ and then we pass $\|\Psi-\bar\Phi_2\|_{L^2_{\mu_0}}\to 0$ to obtain
	 \begin{equation}\label{eq_cp5}
		\liminf_{\|\Psi-\bar\Phi_2\|_{L^2_{\mu_0}}\to 0}\frac{W_2(\Psi,\mu_0)-W_2(\bar\Phi_2,\mu_0)-\frac{2}{\varepsilon}\langle \Psi-\bar\Phi_2,\bar\Phi_1-\bar\Phi_2\rangle_{L^2_{\mu_0}}}{\|\Psi-\bar\Phi_2\|_{L^2_{\mu_0}}}\ge -\varepsilon.
	\end{equation}
	Consequently, $\frac{2}{\varepsilon}(\bar\Phi_1-\bar\Phi_2)\in D^-_\varepsilon W_2(\bar\Phi_2,\mu_0)$, and since $W_2$ is a viscosity supersolution to \eqref{HJB} on $O$, we have furthermore
	\begin{equation}\label{eq_cp6}
		\begin{aligned}
			-\lambda W_2(\bar\Phi_2,\mu_0)+\inf_{u_0\in\mathbb U}\sup_{v_0\in\V}\int_{Z}d\mu_0(z)\big[\ell\big(\bar\Phi_2(z),u_0,v_0\big)\\
			+\frac{2}{\varepsilon}  F(\bar\Phi_2(z),u_0,v_0)\cdot\big( \bar\Phi_1(z)-\bar\Phi_2(z)\big) \big]\le C\varepsilon.
		\end{aligned}
	\end{equation}
	
	In combining both \eqref{eq_cp4} and \eqref{eq_cp6}, one has by applying Lemma \ref{lem_regtech} and the Cauchy-Schwarz inequality:
	\begin{equation}\label{eq_cp7}
		\begin{aligned}
			-2C\varepsilon\le&\lambda (W_2(\bar\Phi_2,\mu_0)-W_1(\bar\Phi_1,\mu_0))\\
			+&\inf_{u_0\in\mathbb U}\sup_{v_0\in\V}\int_{Z}d\mu_0(z)\big[\ell\big(\bar\Phi_1(z),u_0,v_0\big)+\frac{2}{\varepsilon} F(\bar\Phi_1(z),u_0,v_0)\cdot\big(\bar\Phi_1(z)-\bar\Phi_2(z)\big)\big]\\
			-&\inf_{u_0\in\mathbb U}\sup_{v_0\in\V}\int_{Z}d\mu_0(z)\big[\ell\big(\bar\Phi_2(z),u_0,v_0\big)+\frac{2}{\varepsilon} F(\bar\Phi_2(z),u_0,v_0)\cdot\big(\bar\Phi_1(z)-\bar\Phi_2(z)\big) \big]\\
			\le &\lambda (W_2(\bar\Phi_2,\mu_0)-W_1(\bar\Phi_1,\mu_0))
				+L \|\bar\Phi_1-\bar\Phi_2\|_{L^2_{\mu_0}}+\frac{2L}{\varepsilon}\|\bar\Phi_1-\bar\Phi_2\|^2_{L^2_{\mu_0}},
		\end{aligned}
	\end{equation}
	where $L=\max(L_\ell,L_F)$. By \eqref{eq_cp0}, there exists some constant $K>0$ such that 
	$$\|\bar\Phi_1-\bar\Phi_2\|_{L^2_{\mu_0}}\le K\varepsilon.$$
	Hence \eqref{eq_cp7} implies furthermore
	\begin{equation}\label{eq_cp8}
		W_2(\bar\Phi_2,\mu_0)-W_1(\bar\Phi_1,\mu_0)\ge-\frac{(2C+LK+2LK^2)}{\lambda}\varepsilon
	\end{equation}
	which leads to a contradiction to \eqref{eq_cpA} and condition \ref{CP_E1} since one can choose $\varepsilon>0$ arbitrarily small. The proof is complete.
\end{proof}
Now we are ready to state and prove the main result of this section.
\begin{mthm}
	Assume that Assumptions \ref{Assum2} and Isaacs' condition \eqref{eq_IC3} hold. For all $\mu_0\in\Delta(Z)$, the extended values coincide $\mathcal{V}^+(\Phi,\mu_0)=\mathcal{V}^-(\Phi,\mu_0)=:\bar{\mathcal{V}}(\Phi,\mu_0)$ and $\bar{\mathcal{V}}(\Phi,\mu_0)$ is the unique bounded, continuous viscosity solution to \eqref{HJB} on $\mathcal{O}(\mu_0)$ which is Lipschitz continuous in $\Phi$ and verifying the following boundary condition:
	\begin{equation}
		\forall \Phi\in(\mathcal{O}(\mu_0))^c,\ \bar{\mathcal{V}}(\Phi,\mu_0)=\mathcal{V}^+(\Phi,\mu_0).
	\end{equation}
\end{mthm}
\begin{proof}
	The regularity of $\mathcal{V}^\pm$ has been proved in Corollary \ref{cor_reg2} and Lemma \ref{lem_reg4}. It follows from Proposition \ref{prps_subvis} and Corollary \ref{cor_supvis} that $\mathcal{V}^+(\mu_0,\Phi)$ is a viscosity subsolution to \eqref{HJB} on $\mathcal{O}(\mu_0)$, and $\mathcal{V}^-(\mu_0,\Phi)$ is a viscosity supersolution to \eqref{HJB} on the same open set. By Theorem \ref{thm_main1} and Lemma \ref{lem_extvalue}, $\mathcal{V}^\pm(\mu_0,\cdot)$ verifies the boundary condition. Finally, the comparison principle (Proposition \ref{prps_CP}) implies $\mathcal{V}^+(\mu_0,\cdot)=\mathcal{V}^-(\mu_0,\cdot)$ on $\mathcal{O}(\mu_0)$ as well as the uniqueness of viscosity solution verifying the regularity conditions and the boundary condition. The proof is complete.
\end{proof}

\section*{Acknowledgement}
This work was supported by the National Natural Science Foundation of China under Grant No. 12201380.
%--------------------------------------------------------------------------%
\bibliographystyle{plain}
\bibliography{bib_RN}
\end{document}